\documentclass[11pt]{amsart}

\usepackage{mathrsfs}
\usepackage{amsmath,amssymb,amsthm}
\input xy  
\xyoption{all}
\newtheorem{theorem}{Theorem}[section]

\newtheorem{thm}[theorem]{Theorem}
\newtheorem{cor}[theorem]{Corollary}
\newtheorem{lemm}[theorem]{Lemma}
\newtheorem{prop}[theorem]{Proposition}

\theoremstyle{definition}
\newtheorem{definition-theorem}[theorem]{Definition-Theorem}
\newtheorem{defi}[theorem]{Definition}
\newtheorem{remk}[theorem]{Remark}
\newtheorem{exam}[theorem]{Example}
\newtheorem{notation}[theorem]{Notation}

\newcommand{\fftors}{\mbox{\rm f-tors}\hspace{.01in}}

\newcommand{\DD}{\mathsf{D}}
\newcommand{\KKb}{\mathsf{K}^{\rm b}}

\newcommand{\TT}{\operatorname{\mathcal T}\nolimits}

\newcommand{\La}{\Lambda}
\newcommand{\sttilt}{\mbox{\rm s$\tau$-tilt}\hspace{.01in}}

\newcommand{\I}{\widetilde{I}}
\newcommand{\tg}{\tilde{g}}

\newcommand{\wLa}{\widetilde{\Lambda}}
\newcommand{\add}{\mathsf{add}\nolimits}
\newcommand{\proj}{\mathsf{proj}}

\newcommand{\id}{\operatorname{id}\nolimits}

\newcommand{\Hom}{\operatorname{Hom}\nolimits}

\newcommand{\Ext}{\operatorname{Ext}\nolimits}
\newcommand{\Tor}{\operatorname{Tor}\nolimits}

\newcommand{\Soc}{\operatorname{Soc}\nolimits}
\newcommand{\Top}{\operatorname{Top}\nolimits}
\newcommand{\RHom}{\mathbf{R}\strut\kern-.2em\operatorname{Hom}\nolimits}
\newcommand{\HH}{\mathcal{H}}

\def\Im{\mathop{\mathrm{Im}}\nolimits}
\def\Ker{\mathop{\mathrm{Ker}}\nolimits}

\def\Hom{\mathop{\mathrm{Hom}}\nolimits}
\def\End{\mathop{\mathrm{End}}\nolimits}
\def\Ext{\mathop{\mathrm{Ext}}\nolimits}

\def\mod{\mathsf{mod}}

\def\fd{\mathop{\mathsf{fd}}\nolimits}
\def\RHom{\mathop{\mathbb R\mathrm{Hom}}\nolimits}
\newcommand{\Tors}{\mathsf{Tors}\hspace{.01in}}
\newcommand{\Sub}{\mathsf{Sub}\hspace{.01in}}
\newcommand{\Fac}{\mathsf{Fac}\hspace{.01in}}
\def\add{{\mathsf{add}}}
\def\ann{\mathop{\mathrm{ann}}\nolimits}

\def\op{\mathop{\mathrm{op}}\nolimits}

\begin{document}
\title{Classifying $\tau$-tilting modules over preprojective algebras of Dynkin type}
\author{Yuya Mizuno}
\thanks{The author is supported by Grant-in-Aid
for JSPS Fellowships No.23.5593.}
\address{Graduate School of Mathematics\\ Nagoya University\\ Frocho\\ Chikusaku\\ Nagoya\\ 464-8602\\ Japan}
\email{yuya.mizuno@math.nagoya-u.ac.jp}
%\date{\today}

\begin{abstract}
We study support $\tau$-tilting modules over preprojective algebras of Dynkin type. 
We classify basic support $\tau$-tilting modules by giving 
a bijection with elements in the corresponding Weyl groups.
Moreover we show that they are in bijection with the set of torsion classes, the set of torsion-free classes  and other important objects in representation theory. 
We also study $g$-matrices of support $\tau$-tilting modules, which show terms of minimal projective
 presentations of indecomposable direct summands. We give an explicit description of $g$-matrices and prove that cones given by $g$-matrices coincide with chambers of the associated root systems.
\end{abstract}

\maketitle
\tableofcontents

\section*{Introduction}
Preprojective algebras first appeared in the work of Gelfand-Ponomarev \cite{GP}. Since then, they have been one of the important objects in the representation theory of algebras, for example \cite{DR1,BGL,DR2}.  
Preprojective algebras appear also in many branches of mathematics. 
For instance, they play central roles in Lusztig's Lagrangian construction of semicanonical basis \cite{L1,L2}. 
%In this paper we study preprojective algebras from a new perspective. 
%Thus, it is worth studying these important algebras not only from the viewpoint of representation theory, but also other fields.

In \cite{BIRS} (see also \cite{IR}), the authors studied preprojective algebras of non-Dynkin quivers 
via tilting theory. As an application, they succeeded to construct a large class of 2-Calabi-Yau categories which have close connections with cluster algebras. 
On the other hand, in Dynkin cases (i.e. the underlying graph of a quiver is 
$A_n$ ($n\geq1$), $D_n$ ($n\geq 4$) and $E_n$ ($n=6,7,8$)), the preprojective algebras are selfinjective, so that all tilting modules are trivial (i.e. projective) and they are not effective tools. 
In this paper, instead of tilting modules, we will study support $\tau$-tilting modules (Definition \ref{spt}) over the preprojective algebras of Dynkin type. 

The notion of support $\tau$-tilting modules was introduced in \cite{AIR}, as a generalization of tilting modules. 
Support $\tau$-tilting modules have several nice properties. For example,
it is shown that there are deep connections between $\tau$-tiltihg theory, torsion theory, silting theory and cluster tilting theory. 
Support $\tau$-tilting modules also have nicer mutation properties than tilting modules, that is, mutation of support $\tau$-tilting modules is always possible (Definition-Theorem \ref{air mutation}). 
Moreover, support $\tau$-tilting modules over selfinjective algebras are useful to provide tilting complexes \cite{M}. 
It is therefore fruitful to investigate these remarkable modules for preprojective algebras. 
To explain our results, we give the following set-up. 

Let $Q$ be a finite connected acyclic quiver with the set $Q_0=\{1,\ldots, n\}$ of vertices, $\La$ the preprojective algebra of $Q$ and $I_i$ the two-sided ideal of $\La$ generated by $1-e_i$, where $e_i$ is an idempotent of $\La$ corresponding to $i\in Q_0$. 
These ideals are quite useful to study structure of categories \cite{IR,BIRS,AIRT,ORT}. 
They also play important roles in Geiss-Leclerc-Schr{\"o}er's construction of cluster monomials of certain types of cluster algebras \cite{GLS1,GLS3} and Baumann-Kamnitzer-Tingley's theory of affine MV polytopes \cite{BKT}.  

If $Q$ is non-Dynkin, then $I_i$ is a tilting module. 
On the other hand, if $Q$ is Dynkin, then $I_i$ is never a tilting module. 
Instead, we show that $I_i$ is a support $\tau$-tilting module (Lemma \ref{I_i}) and this observation is crucial in this paper. 
We denote by $\langle I_1,\ldots,I_n\rangle$ the set of ideals of $\La$ of the form $I_{i_1}I_{i_2}\cdots I_{i_k}$ for some $k\geq0$ and $i_1,\ldots,i_k\in Q_0$. 

Our main goal is to give the following bijections which 
provide a complete description of support $\tau$-tilting modules over preprojective algebras of Dynkin type.

\begin{thm}[Theorem \ref{number}]\label{2}
Let $Q$ be a Dynkin quiver with the set $Q_0=\{1,\ldots, n\}$ of vertices and $\La$ the preprojective algebra of $Q$. 
There are bijections between the following objects.
\begin{itemize}
\item[(a)]The elements of the Weyl group $W_Q$ associated with $Q$.
\item[(b)]The set $\langle I_1,\ldots,I_n\rangle$.
\item[(c)]The set $\sttilt\La$ of isomorphism classes of basic support $\tau$-tilting $\La$-modules.
\item[(d)]The set $\sttilt\La^{\op}$ of isomorphism classes of basic support $\tau$-tilting $\La^{\op}$-modules.
\end{itemize}
\end{thm}

These bijections give a strong relationship between the representation theory of algebras and root systems as below. 
It is known that support $\tau$-tilting modules have a natural structure of partially ordered set (Definition \ref{def partial}). 
We relate this partial order with that of the Weyl group.

\begin{thm}[Theorem \ref{main3}]
The bijection $W_Q\to\sttilt\La$ in Theorem \ref{2} gives an isomorphism of partially ordered sets 
$(W_Q,\leq_L)$ and $(\sttilt\La,\leq)^{\op}$, where $\leq_L$ is the left weak order on $W_Q$.\end{thm}

We also study the $g$-\emph{matrix} (Definition \ref{def g-vector}) of support $\tau$-tilting module $T$, which shows terms of minimal projective presentations of each indecomposable direct summand of $T$.  
By $g$-matrices, we define \emph{cones}, which provide a geometric realization of the simplicial complex of support $\tau$-tilting modules. 
We give a complete description of $g$-matrices in terms of reflection transformations and show that their cones coincide with chambers of the associated root system. 
The results are summarized as follows, where ${\bf e}_1,\ldots,{\bf e}_n$ is the canonical basis of $\mathbb{Z}^n$, $C_0:=\{a_1{\bf e}_1+\cdots+a_n{\bf e}_n \ |\ a_i\in \mathbb{R}_{>0}\}$ and $\sigma^*: W_Q\rightarrow {GL}_n(\mathbb{Z})$ is the contragradient of the geometric representation (Definition \ref{geometric representation}). 

\begin{thm}[Theorem \ref{weyl group}]
Let $Q$ be a Dynkin quiver with the set $Q_0=\{1,\ldots, n\}$ of vertices and $\La$ the preprojective algebra of $Q$. 
\begin{itemize}
\item[(a)] The set of $g$-matrices of support $\tau$-tilting $\La$-modules coincides with $\sigma^*(W_Q)$. 
%the subgroup $\langle \sigma^*_1,\ldots,\sigma^*_n\rangle$ of $\textrm{GL}(\mathbb{Z}^n)$ generated by $\sigma^*_i$ for all $i\in Q_0$. 
\item[(b)]For any $w\in W_Q$, we have 
$$C(w)=\sigma^*(w)(C_0),$$
where $C(w)$ denotes the cone of the support $\tau$-tilting module $I_w$ $($Definition \ref{def g-vector}$)$.
In particular, cones of basic support $\tau$-tilting $\La$-modules give chambers of the associated root system.
\end{itemize}
\end{thm}

The above results give several important applications. 
Among others we conclude that there are only finitely many basic support $\tau$-tilting $\La$-modules. 
This fact yields bijections between basic support $\tau$-tilting $\La$-modules, torsion classes in $\mod\La$ and torsion-free classes in $\mod\La$ by a result in \cite{IJ}. Then we can give complete descriptions of these two classes in terms of $W_Q$ (Proposition \ref{bij torsion}). 
This seems to be interesting by itself since almost all preprojective algebras are of wild representation type.

As another application, by combining with results of \cite{AIR}, \cite{KY,BY} and \cite{ORT}, we extend bijections of Theorem \ref{2} as follows. 

\begin{cor}[Corollary \ref{further bijections}]\label{intro}
Objects of Theorem \ref{2} are in bijection with the following objects.
\begin{itemize}
\item[(e)]The set of torsion classes in $\mod\La$.
\item[(f)]The set of torsion-free classes in $\mod\La$.
\item[(g)]The set of isomorphism classes of basic two-term silting complexes in $\KKb(\proj\La)$.
\item[(h)]The set of intermediate bounded co-$t$-structures in $\KKb(\proj\La)$ with respect to the standard co-$t$-structure.
\item[(i)]The set of intermediate bounded $t$-structures in $\DD^{\rm b}(\mod\La)$ with length heart with respect to the standard $t$-structure. 
\item[(j)]The set of isomorphism classes of two-term simple-minded collections in $\DD^{\rm b}(\mod\La)$.
\item[(k)]The set of quotient closed subcategories in $\mod KQ$.
\item[(l)]The set of subclosed subcategories in $\mod KQ$.
\end{itemize}
\end{cor}

We now describe the organization of this paper.

In section \ref{preli}, we recall definitions of preprojective algebras, support $\tau$-tilting modules and torsion classes.

In section \ref{spt ideal}, we introduce support $\tau$-tilting ideals (Definition \ref{def spt}), which play key roles in this paper. 
Moreover, we prove that any element of $\langle I_1,\ldots,I_n\rangle$ is a basic support $\tau$-tilting ideal (Theorem \ref{main1}). 
In section \ref{bijection with the Weyl group}, we prove that elements of $W_Q$ correspond bijectively with support $\tau$-tilting modules (Theorem \ref{number}). 
Using this, we give a description of mutations of support $\tau$-tilting modules in terms of the Weyl group $W_Q$.  
We also determine the annihilators of all support $\tau$-tilting modules. 
In section \ref{Partial orders of}, 
we give an (anti)isomorphism between a partial order of $W_Q$ and that of $\sttilt\La$ (Theorem \ref{main3}). 

In section \ref{Cone of $g$-vectors}, we deal with $g$-matrices of support $\tau$-tilting modules. We give an explicit description of them in terms of reflection transformations. Furthermore, we study cones of support $\tau$-tilting modules and show that they coincide with chambers of the associated root system.

In section \ref{further}, we show Corollary \ref{intro}. In particular, we give complete descriptions of torsion classes and torsion-free classes (Proposition  \ref{bij torsion}). 
Furthermore, we explain a relationship between our results and other works.
\\

\textbf{Notations.}
Let $K$ be an algebraically closed field and we denote by $D:=\Hom_K(-,K)$. 
By a finite dimensional algebra $\La$, we mean a basic finite dimensional algebra over $K$. 
By a module, we mean a right module unless stated otherwise. 
We denote by $\mod\Lambda$ the category of finitely generated $\Lambda$-modules and by $\proj\La$ the category of finitely generated projective $\La$-modules. The composition $gf$ means first $f$, then $g$. We denote by $Q_0$ the set of vertices and by $Q_1$ the set of arrows of a quiver $Q$. For an arrow $a\in Q_1$, we denote by $s(a)$ and $e(a)$ the start and end vertices of $a$ respectively. 
For $X\in\mod\La$, we denote by $\Sub X$ (respectively, $\Fac X$) the subcategory of $\mod\La$ whose objects are submodules (respectively, factor modules) of finite direct sums of copies of $X$. 
We denote by $\add M$ the subcategory of $\mod\Lambda$ consisting of direct summands of finite direct sums of copies of $M\in\mod\La$. 
\\ 

\textbf{Acknowledgement.}
First and foremost, the author would like to thank Osamu Iyama for his support and patient guidance. He would like to express his gratitude to Steffen Oppermann, who kindly explain results of his paper.
He is grateful to Joseph Grant, Laurent Demonet and Dong Yang for answering  
questions and helpful comments. He thanks Kota Yamaura, Takahide Adachi and Gustavo Jasso for their help and stimulating discussions. The author is very grateful to the anonymous referee for valuable comments, especially for suggesting the better terms and sentences.

%%%%%%%%%%%%%%%%%%%%%%%%%%%%%%%%%%%%%%%%%%%%%%%%%%%%%%%%%%%%%%%%%%%%%%%%%%%%%%%%%%%% 
\section{Preliminaries}\label{preli}
%%%%%%%%%%%%%%%%%%%%%%%%%%%%%%%%%%%%%%%%%%%%%%%%%%%%%%%%%%%%%%%%%%%%%%%%%%%%%%%%%%%% 

\subsection{Preprojective algebras}
In this subsection, we recall definitions and some properties of preprojective algebras. 
We refer to \cite{BBK,Ri} for basic properties and background information. 

\begin{defi}\label{preprojective}
Let $Q$ be a finite connected acyclic quiver with vertices $Q_0=\{1,\ldots, n\}$. 
The preprojective algebra associated to $Q$ is the algebra
$$\La=K\overline{Q}/\langle \sum_{a\in Q_1} (aa^*-a^*a)\rangle$$
where $\overline{Q}$ is the double quiver of $Q$, which is obtained from $Q$ by adding for each arrow 
$a:i\rightarrow j$ in $Q_1$ an arrow $a^*:i\leftarrow j$ pointing in the opposite direction. 
\end{defi}

%If $Q$ is a Dynkin quiver, we denote by $\La$ the preprojective algebra of $Q$ and if $Q$ is a non-Dynkin quiver, we denote by $\La$ the completed preprojective algebra of $Q$. 
If $Q$ is a non-Dynkin quiver, 
the bounded derived category $\DD^{\rm b}(\fd\Lambda)$
of the category $\fd\La$ of finite dimensional modules are \emph{2-Calabi-Yau} (2-CY for short), that is, there is a functorial isomorphism  $$D{\Hom}_{\DD^{\rm b}(\fd\Lambda)}(X,Y)\cong{\Hom}_{\DD^{\rm b}(\fd\Lambda)}(Y,X[2]).$$
We refer to \cite{AR2,Boc,CB} and see \cite{GLS2} for the detailed proofs. 
On the other hand, the algebra $\Lambda$ is finite-dimensional selfinjective if $Q$ is a Dynkin quiver. 

%In this case, the stable category $\underline{\mod}\Lambda$ is 2-CY 

In this paper, we use the two-sided ideal $I_i$ of $\La$ generated by $1-e_i$, where $e_i$ is a primitive idempotent of $\La$ for $i\in Q_0$. It is easy to see that the ideal has the following property, where we denote by $S_i:=\La/I_i$.
\begin{lemm}\label{str of I}
Let $Q$ be a finite connected acyclic quiver, $\La$ the preprojective algebra of $Q$ and $X\in\mod\La$.
Then $XI_i$ is maximal amongst submodules $Y$ of $X$ such that any composition factor of $X/Y$ is isomorphic to a simple module $S_i$.
\end{lemm}

\subsection{Support $\tau$-tilting modules}\label{Support tau tilting modules}
We recall the definition of support $\tau$-tilting modules.  
We refer to \cite{AIR} for the details about support $\tau$-tilting modules.

Let $\La$ be a finite dimensional algebra and $\tau$ denote the AR translation \cite{ARS}.

\begin{defi}\label{spt}
\begin{itemize}
\item[(a)] We call $X$ in $\mod\Lambda$ \emph{$\tau$-rigid} if $\Hom_{\Lambda}(X,\tau X)=0$.
\item[(b)] We call $X$ in $\mod\Lambda$ \emph{$\tau$-tilting} (respectively, \emph{almost complete $\tau$-tilting}) if $X$ is $\tau$-rigid and $|X|=|\Lambda|$ (respectively, $|X|=|\Lambda|-1$), where $|X|$ denotes the number of non-isomorphic indecomposable direct summands of $X$.

\item[(c)] We call $X$ in $\mod\Lambda$ \emph{support $\tau$-tilting} if there exists an idempotent $e$ of $\La$ such that $X$ is a $\tau$-tilting $(\La/\langle e\rangle)$-module.
\item[(d)] We call a pair $(X,P)$ of $X\in\mod\La$ and $P\in\proj\Lambda$ 
 \emph{$\tau$-rigid} if $X$ is $\tau$-rigid and $\Hom_\Lambda(P,X)=0$.
\item[(e)] We call a $\tau$-rigid pair $(X,P)$ a \emph{support $\tau$-tilting} (respectively, \emph{almost complete support $\tau$-tilting}) pair if $|X|+|P|=|\Lambda|$ (respectively, $|X|+|P|=|\Lambda|-1$). 
\end{itemize}
\end{defi}

We call $(X,P)$ \emph{basic} if $X$ and $P$ are basic, and we say that $(X,P)$ is a \emph{direct summand} of $(X',P')$ if $X$ is a direct summand of $X'$ and $P$ is a direct summand of $P'$. 
Note that $(X,P)$ is a $\tau$-rigid (respectively, support $\tau$-tilting) pair for $\Lambda$ if and only if
$X$ is a $\tau$-rigid (respectively, $\tau$-tilting) $(\Lambda/\langle e\rangle)$-module, where $e$ is an idempotent of $\Lambda$ such that $\add P=\add e\Lambda $ \cite[Proposition 2.3]{AIR}. Moreover, if $(X,P)$ and $(X,P')$ are support $\tau$-tilting pairs for $\La$, then we get $\add P=\add P'$. Thus, a basic support $\tau$-tilting module $X$ determines a basic support $\tau$-tilting pair $(X,P)$ uniquely and we can identify basic support $\tau$-tilting modules with basic support $\tau$-tilting pairs. 

We denote by $\sttilt\La$ the set of isomorphism classes of basic support $\tau$-tilting $\La$-modules.

Next we recall some properties of support $\tau$-tilting modules. 
The set of support $\tau$-tilting modules has a natural partial order as follows.

\begin{defi}\label{def partial}
Let $\La$ be a finite dimensional algebra. 
For $T,T'\in\sttilt\Lambda$, we write
\[T'\geq T\]
if $\Fac T' \supset \Fac T$. 
Then $\geq$ gives a partial order on $\sttilt\Lambda$ \cite[Theorem 2.18]{AIR}. 
\end{defi}

Then we give the following results, which play significant roles in this paper. 

\begin{definition-theorem}\cite[Theorem 2.18, 2.28 and 2.30]{AIR}\label{air mutation}
Let $\La$ be a finite dimensional algebra. Then
\begin{itemize}
\item[(i)] any basic almost support $\tau$-tilting pair $(U,Q)$ is a direct summand of exactly two basic support $\tau$-tilting pairs $(T,P)$ and $(T',P')$. Moreover, we have $T> T'$ or $T< T'$.
\end{itemize}

Under the above setting, let $X$ be an indecomposable $\La$-module satisfying either $T=U\oplus X$ or $P=Q\oplus X$. 
We write $(T',P')=\mu_{(X,0)}(T,P)$ if $X$ is a direct summand of $T$ and $(T',P')=\mu_{(0,X)}(T,P)$ if $X$ is a direct summand of $P$, and we say that $(T',P')$ is a \emph{mutation} of $(T,P)$. In particular, we say that $(T',P')$ is a \emph{left mutation} (respectively, \emph{right mutation}) of $(T,P)$ if $T>T'$ (respectively, if $T<T'$) and write $\mu^-$ (respectively, $\mu^+$). By (i), exactly one of the left mutation or right mutation occurs.  

Now, assume that $X$ is a direct summand of $T$ and $T=U\oplus X$. In this case, for simplicity,  
we write a left mutation $T'=\mu_X^-(T)$ and a right mutation $T'=\mu_X^+(T)$.

Then, the following conditions are equivalent.
\begin{itemize}
\item[(ii)]
\begin{itemize}
\item[(a)] $T>T'$ (i.e. $T'=\mu_X^-(T)$).
\item[(b)] $X\notin\Fac U$.
\end{itemize}
\end{itemize}

%Note that, if $P=Q\oplus X$, then $\mu_{(0,X)}(T,P)$ is always right mutation.

Furthermore, if $T$ is a $\tau$-tilting $\La$-module and (ii) holds, then we obtain the following result.
\begin{itemize}
\item[(iii)]
Let $X\overset{f}{\to} U'\to Y\to 0$ be an exact sequence, where $f$ is a minimal left $(\add U)$-approximation.
Then one of the following holds.
\begin{itemize}
\item[(a)] $Y=0$ and $\mu_X^-(T)\cong U$ is a basic support $\tau$-tilting $\Lambda$-module. 
\item[(b)] $Y\neq0$ and $Y$ is a direct sum of copies of an indecomposable $\Lambda$-module $Y_1$, which is not in $\add T$, and $\mu_X^-(T)\cong Y_1\oplus U$ is a basic $\tau$-tilting $\Lambda$-module.
\end{itemize}
\end{itemize}
\end{definition-theorem}

Finally, we define the following useful quiver.

\begin{defi}
Let $\La$ be a finite dimensional algebra.
We define the \emph{support $\tau$-tilting quiver} $\HH(\sttilt\La)$ as follows.

$\bullet$ The set of vertices is $\sttilt\La$.

$\bullet$ Draw an arrow from $T$ to $T'$ if $T'$ is a left mutation of $T$ (i.e. 
 $T'=\mu_X^-(T)$).
%(or equivalently, $T$ is a right  mutation of $T'$).
\end{defi}

The following theorem relates $\HH(\sttilt\La)$ with partially orders of $\sttilt\La$.

\begin{theorem}\label{Hasse-mut}\cite[Corollary 2.34]{AIR}
The support $\tau$-tilting quiver $\HH(\sttilt\La)$ is the Hasse quiver of the partially ordered set $\sttilt\La$.
\end{theorem}

\subsection{Torsion classes}
The notion of torsion classes has been extensively studied in the representation theory.  As we will see below, support $\tau$-tilting modules have a close relationship with torsion classes.  

\begin{defi}Let $\La$ be a finite dimensional algebra and $\TT$ a full subcategory in $\mod\Lambda$.
\begin{itemize}
\item[(a)]We call $\TT$ \emph{torsion class} (respectively, \emph{torsion-free class}) if it is closed under factor modules (respectively, submodules) and extensions.

\item[(b)]We call $\TT$ a \emph{contravariantly finite subcategory} in $\mod\La$ if for each $X\in\mod\La$, there is a map $f:T\to X$ with $T\in\TT$ such that $\Hom_\La(T',T)\overset{\cdot f}{\to}\Hom_\La(T',X)$ is surjective for all $T'\in\TT$.
Dually, a \emph{covariantly finite subcategory} is defined. 
We call $\TT$ a \emph{functorially finite subcategory} if it is contravariantly and covariantly finite.
\end{itemize}
\end{defi}

We denote by $\fftors\Lambda$ the set of functorially finite torsion classes in $\mod\Lambda$.

Then, we have the following result, where we denote by $P(\TT)$ the direct sum of one copy of each of the indecomposable \emph{Ext-projective} objects in $\TT$ (i.e. $\{X\in\TT\ |\ \Ext_\La^1(X,\TT)=0\}$) up to isomorphism. 

\begin{theorem}\cite{AIR}\label{fftor-spt}
For a finite dimensional algebra $\La$, there are mutually inverse bijections 
\[\sttilt\Lambda\longleftrightarrow\fftors\Lambda\]
given by $\sttilt\La\ni T\mapsto\Fac T\in\fftors\Lambda$ and $\fftors\Lambda\ni\TT\mapsto P(\TT)\in\sttilt\Lambda$.
\end{theorem}

Thus, support $\tau$-tilting modules are quite useful to investigate 
functorially finite torsion classes. Moreover, we have the following much stronger result  under the assumption that $|\sttilt \La|<\infty$.

\begin{thm}\cite{IJ}\label{iyama}
For a finite dimensional algebra $\La$, the following conditions are equivalent. 
\begin{itemize}
\item[(a)] $|\sttilt \La|<\infty$. 
\item[(b)] Any torsion class in $\mod\La$ is functorially finite.
\item[(c)] Any torsion-free class in $\mod\La$ is functorially finite.
\end{itemize}
\end{thm}

Therefore, if there are only finitely many basic support $\tau$-tilting modules, then Theorem \ref{fftor-spt} give a bijection between $\sttilt\La$ and torsion classes of $\mod\La$. 

%%%%%%%%%%%%%%%%%%%%%%%%%%%%%%%%%%%%%%%%%%%%%%%%%%%%%%%%%%%%%%%%%%%%%%%%%%%%%%%%%%%%%%%%%%%%
\section{Support $\tau$-tilting ideals and the Weyl group}\label{connection with the Weyl group}
The aim of this section is to give a complete description of all support $\tau$-tilting modules 
over preprojective algebras of Dynkin type. We also study their several properties.

Throughout this section, unless otherwise specified, let $Q$ be a Dynkin quiver with $Q_0=\{1,\ldots, n\}$, $\La$ the preprojective algebra of $Q$ and $I_i:=\La(1-e_i)\La$ for $i\in Q_0$. 
We denote by $\langle I_1,\ldots,I_n\rangle$ the set of ideals of $\La$ which can be written as 
$$I_{i_1}I_{i_2}\cdots I_{i_k}$$ for some $k\geq0$ and $i_1,\ldots,i_k\in Q_0$.

\subsection{Support $\tau$-tilting ideals}\label{spt ideal}
In this subsection, we introduce support $\tau$-tilting ideals.  
We will show that any element of  $\langle I_1,\ldots,I_n\rangle$ is a basic support $\tau$-tilting ideal. This fact plays a key role in this paper. 

We start with the following definition.

\begin{defi}\label{def spt} 
We call a two-sided ideal $I$ of $\La$ {\em support $\tau$-tilting} if $I$ is a left support $\tau$-tilting $\La$-module and a right support $\tau$-tilting $\La$-module. 
\end{defi}

The aim of this subsection is to prove the following result.

\begin{thm}\label{main1}
Any $T\in\langle I_1,\ldots,I_n\rangle$ is a basic support $\tau$-tilting ideal of $\La$.
\end{thm}

First, we recall some properties for later use (see \cite{BBK}). 
For any $i\in Q_0$, we can take a minimal projective presentation of $S_i:=\La/I_i$ as follows

\begin{eqnarray*}\label{simple resolution}
\xymatrix@C20pt@R20pt{ e_i\Lambda\ar[r]^{p_2\ \ \ \ \ \ \ \ \ \ }  &{\displaystyle\bigoplus_{\begin{smallmatrix}a\in \overline{Q}_1, e(a)=i\end{smallmatrix}} e_{s(a)}\La}  \ar[r]^{\ \ \ \ \ \ \ \ \ \ p_1 } & e_i\Lambda \ar[r]^{p_0}&S_i\ar[r]&0.     }
\end{eqnarray*}
Since $\Im p_1=e_iI_i$, we have the following exact sequences
\begin{equation}\label{eI s_i}
\xymatrix@C20pt@R20pt{ 0 \ar[r]^{}  & e_iI_i \ar[r]^{\iota}  &e_i\Lambda \ar[r]^{p_0} & S_i \ar[r]&0.     }
\end{equation}
\begin{equation}\label{La I}
\xymatrix@C20pt@R20pt{  e_i\Lambda \ar[r]^{p_2\ \ \ \ \ \ \ \ }  &{\displaystyle\bigoplus_{\begin{smallmatrix}a\in \overline{Q}_1, e(a)=i\end{smallmatrix}} e_{s(a)}\La}   \ar[r]^{\ \ \ \ \ \ \ \ \pi } &e_iI_i \ar[r]&0.     }\end{equation}
As a direct sum of (\ref{eI s_i}) and $0\to(1-e_i)\La\stackrel{\id}{\to}(1-e_i)\La\to0$, we have an exact sequence

\begin{equation}\label{I s_i}
\xymatrix@C20pt@R20pt{ 0 \ar[r]^{}  & I_i \ar[r]^{\left(\begin{smallmatrix}\iota& 0\\0&\id \end{smallmatrix}\right)}  &\Lambda  \ar[r]^{(p_0\ 0) } & S_i \ar[r]&0.     }
\end{equation}

Moreover, we have $\Ker p_2=S_{\sigma(i)}$, where $\sigma:Q_0\to Q_0$ is a Nakayama permutation of $\La$ (i.e. $D(\La e_{\sigma(i)})\cong e_{i}\La$) \cite[Proposition 4.2]{BBK}. 
Therefore,  we have the following exact sequences

\begin{equation}\label{simple sigma resolution}
\xymatrix@C20pt@R20pt{ 0 \ar[r]^{}  & S_{\sigma(i)} \ar[r]^{}  &e_i\Lambda\ar[r]^{p_2 \ \ \  \ \ \ \ \ \ \ } &{\displaystyle\bigoplus_{\begin{smallmatrix}a\in \overline{Q}_1, e(a)=i\end{smallmatrix}} e_{s(a)}\La}.     }
\end{equation}

Then we give the following lemma.

\begin{lemm}\label{I_i}
$I_i$ is a basic support $\tau$-tilting ideal of $\La$ for any $i\in Q_0$. 
\end{lemm}

\begin{proof}
It is clear that $I_i$ is basic. 
We will show that $I_i$ is a right support $\tau$-tilting $\La$-module; the proof for a left $\La$-module is similar. 
%Similarly we can show for $\La^{op}$-modules. 
If $e_i\La$ is a simple module, then $\La$ is a simple algebra since $\La$ is a preprojective algebra. 
Thus, in this case, we have $I_i=0$ and it is a support $\tau$-tilting module.

Assume that $e_i\La$ is not a simple module and hence $e_iI_i\neq0$. 
Since $\La$ is selfinjective, $e_j\La$ has a simple socle for any $j\in Q_0$ and these simple modules are mutually non-isomorphic. Therefore, $e_iI_i$ is not isomorphic to $e_j\La$ for any $j\in Q_0$. 
Hence, we obtain $|I_i|=|\La|$. Thus, in this case, it is enough to show that 
$\Hom_\La(I_i,\tau(I_i))=0.$ 

Applying the functor $\nu:=D\Hom_\La(-,\La)$ to (\ref{La I}), 
we have the following exact sequence 

\begin{equation}\label{tau I_i}
\xymatrix@C25pt@R20pt{ 0 \ar[r]^{}  & \tau(e_iI_i) \ar[r]^{}  & \nu(e_i\Lambda)\ar[r]^{\nu p_2\ \ \ \ \ \ \ \ \ \ } &{\displaystyle\bigoplus_{\begin{smallmatrix}a\in \overline{Q}_1, e(a)=i\end{smallmatrix}} \nu(e_{s(a)}\La).}     }
\end{equation}

On the other hand, applying the functor $\nu$ to (\ref{simple sigma resolution}), we have the following exact sequence 

\begin{equation}\label{sigma resolution}
\xymatrix@C20pt@R20pt{ 0 \ar[r]^{}  & \nu(S_{\sigma(i)}) \ar[r]^{}  & \nu(e_i\Lambda)\ar[r]^{\nu p_2\ \ \ \ \ \ \ \ \  } &{\displaystyle\bigoplus_{\begin{smallmatrix}a\in \overline{Q}_1, e(a)=i\end{smallmatrix}} \nu(e_{s(a)}\La)}.     }
\end{equation}

Since $\La$ is selfinjective, we have $\nu(S_{\sigma(i)})\cong\nu(\Soc(e_i\La))\cong\Top(e_i\La)\cong S_i$. 
Therefore, comparing (\ref{tau I_i}) with (\ref{sigma resolution}), 
we get $\tau(I_i)\cong\tau(e_iI_i)\cong S_i$. 
Since $S_i\notin\add(\Top(I_i))$, we have $\Hom_\La(I_i,\tau(I_i))=0$. 
\end{proof}

The next two statements are analogous to results of \cite[Lemma II.1.1, Proposition II.1.5]{BIRS}, where tilting modules over preprojective algebras of non-Dynkin type are treated. 
Though the proofs given here are essentially the same, we include them for  convenience.

\begin{lemm}\label{2of1}
Let $T$ be a support $\tau$-tilting $\Lambda$-module. 
For a simple $\La^{\op}$-module $S$, at least one
of the statements $T\otimes_\Lambda S=0$ and ${\Tor}^\Lambda_1(T,S)=0$ holds.
\end{lemm}

\begin{proof}
By \cite[Proposition 2.5]{AIR}, we can take a minimal projective presentation 
$P_1\to P_0\to T\to0$ such that $P_0$ and $P_1$ do not have a common summand. 
Thus, at least one of the statement $P_0\otimes_\Lambda S=0$ and $P_1\otimes_\Lambda S=0$ holds.
Since we have $\Tor_1^{\La}(T,S)\cong P_1\otimes_\Lambda S$ and $T\otimes_\Lambda S\cong P_0\otimes_\Lambda S$,  the assertion follows. 
\end{proof}

\begin{cor}\label{tensor}
Let $T$ be a support $\tau$-tilting $\Lambda$-module. 
If $TI_i\neq T$, then we have $T\otimes_{\La}I_i\cong TI_i$.
\end{cor}

\begin{proof}
By the assumption $TI_i\neq T$ and Lemma \ref{str of I}, we have $T\otimes_\Lambda S_i\neq 0$.
Then, by Lemma \ref{2of1}, we obtain ${\Tor}^\Lambda_1(T,S_i)=0$. 
By applying the functor $T\otimes_\La-$ to the exact sequence (\ref{I s_i}), 
we get the following exact sequence 
$$0={\Tor}^\Lambda_1(T,S_i)\to T\otimes_\Lambda I_i\to T\otimes_\Lambda\La=T.$$
Hence we have $T \otimes_{\La} I_i\cong\Im( T\otimes_\Lambda I_i\to T)=TI_i$.
\end{proof}

Next we give the following easy observation.

\begin{lemm}\label{decompose}
Let $I$ be a two-sided ideal of $\La$. 
For a primitive idempotent $e_i$ with $i\in Q_0$, 
$e_iI$ is either indecomposable or zero. 
\end{lemm}

\begin{proof}
Since $\La$ is selfinjective, $e_i\La$ has a simple socle. 
Thus the submodule $e_iI$ of $e_i\La$ has a simple socle or zero. 
\end{proof}

The next two lemmas are crucial.

\begin{lemm}\label{surje}
Let $I$ be a two-sided ideal of $\La$. 
Then there exists a surjective map  
$$\La\to\End_\La(I),\ \ \lambda\mapsto(I\ni x\mapsto \lambda x\in I).$$
%[[Moreover it is a $\La$-module map. ]]
In particular, for an idempotent $e_i$, $i\in Q_0$, we have a surjective map 
$$ e_i\La(1-e_i)\to\Hom_\La((1-e_i)I,e_iI).$$ 
\end{lemm}

\begin{proof}
We have to show that, for any $f\in\End_\La(I)$, there exists $\lambda\in\La$ such that $f=(\lambda\cdot):I\to I.$ 
Since $\La$ is selfinjective, there exists a $\La$-module map $h:\La\to\La$ 
making the diagram 
\[\xymatrix{0\ar[r]&I\ar[r]^{\iota}\ar[d]^{f}&\La\ar@{.>}[d]^h\\
0\ar[r]&I\ar[r]^{\iota}&\La,}
\]
commutative, where $\iota$ is the canonical inclusion and $h:\La\to\La$ is the $\La$-module map. 
Put $h(1)=\lambda\in\La$. Then $h$ is given by left multiplication with $\lambda$.  
Thus, we obtain $f(x)=h(x)=\lambda x$ for any $x\in I$ and we have proved our claim. 
The second statement easily follows from the first one. 
\end{proof}

\begin{lemm}\label{bongartz}
Let $T$ be a support $\tau$-tilting ideal of $\La$. 
If $I_iT\neq T$, then we have $e_iT\notin\Fac((1-e_i)T)$. 
In particular, $T$ has a left mutation $\mu_{e_iT}^-(T)$.
\end{lemm}

\begin{proof} 
By Lemma \ref{decompose} and the assumption, $e_iT\neq0$ is indecomposable. 
We assume that $e_iT\in\Fac((1-e_i)T)$. 
Then, there exist an integer $d>0$ and a surjective map  
$$\{(1-e_i)T\}^{d}\to e_iT.$$
%where $\{(1-e_i)T\}^{n}$ denote by the $n$-times direct sum of $X\in\mod\La$ for $n\in\mathbb{Z}$.

By Lemma \ref{surje}, there exist $\lambda_m\in\La$ $(1\leq m\leq d)$ such that  
$$\sum_{m=1}^{d}e_i\lambda_m(1-e_i)\cdot(1-e_i)T= e_iT.$$ 

On the other hand, since $I_i=\La(1-e_i)\La$, we get 
$$e_i\La(1-e_i)\cdot(1-e_i)T=e_iI_iT.$$ 
Thus we have  
$$e_iT=\sum_{m=1}^{d}e_i\lambda_m(1-e_i)T\subset e_iI_iT,$$ 
which is a contradiction. 
The second statement given in this lemma follows from Definition-Theorem \ref{air mutation} (ii). 
\end{proof}

We also need to prove the following statement.

\begin{lemm}\label{approx}
Let $T$ be a support $\tau$-tilting ideal of $\La$. For the map $p_2$ in 
the sequence $(\ref{La I})$, 
the map  
$$\xymatrix@C30pt@R20pt{ e_i\La\otimes_\La T \ar[r]^{p_2\otimes_\La T\ \ \ \ \ \ \ \ \ }  & {\displaystyle\bigoplus_{\begin{smallmatrix}a\in \overline{Q}_1, e(a)=i\end{smallmatrix}} e_{s(a)}\La\otimes_\La T} }$$
 is a left $(\add((1-e_i)T))$-approximation.
\end{lemm}

\begin{proof}
We will show that the map  
\begin{eqnarray}\label{surjection}
\ \ \ \ \  \Hom_\La(\bigoplus_{\begin{smallmatrix}a\in \overline{Q}_1, e(a)=i\end{smallmatrix}} e_{s(a)}\La\otimes_\La T,(1-e_i)T)\overset{(p_2\otimes_\La T,(1-e_i)T)}{\longrightarrow}
\Hom_\La(e_i\La\otimes_\La T,(1-e_i)T)
\end{eqnarray}
is surjective.

Let $E:=\End_\La(T)$. 
Then, we have 
\begin{eqnarray*}
\Hom_\La(e_i\La\otimes_\La T,(1-e_i)T) &\cong&
\Hom_\La(e_i\La,\Hom_\La(T,(1-e_i)T))\\
&\cong& \Hom_\La(e_i\La,(1-e_i)E), 
\end{eqnarray*}
and similarly $$\Hom_\La(\bigoplus_{\begin{smallmatrix}a\in \overline{Q}_1, e(a)=i\end{smallmatrix}} e_{s(a)}\La\otimes_\La T,(1-e_i)T)\cong\Hom_\La(\bigoplus_{\begin{smallmatrix}a\in \overline{Q}_1, e(a)=i\end{smallmatrix}} e_{s(a)}\La,(1-e_i)E).$$ 
Then, from the functoriality, 
we write the map (\ref{surjection}) as follows.
\begin{eqnarray*}\Hom_\La(\bigoplus_{\begin{smallmatrix}a\in \overline{Q}_1, e(a)=i\end{smallmatrix}} e_{s(a)}\La,(1-e_i)E)\overset{(p_2,(1-e_i)E)}{\longrightarrow}
\Hom_\La(e_i\La,(1-e_i)E).
\end{eqnarray*}

On the other hand, by Lemma \ref{surje}, there exists a surjective $\La$-module map  
$$g:\La\longrightarrow E.$$
%Note that it is a $\La$-module map. 
Then, we have the following commutative diagram

\[\xymatrix@C80pt@R40pt{\Hom_\La({\displaystyle\bigoplus_{\begin{smallmatrix}a\in \overline{Q}_1, e(a)=i\end{smallmatrix}} e_{s(a)}\La},(1-e_i)\La)\ar@{>>}[r]^{\ \ \ \ \ \ \ \ \ (p_2,(1-e_i)\La)}\ar@{>>}[d]^{(\bigoplus_{\begin{smallmatrix}a\in \overline{Q}_1, e(a)=i\end{smallmatrix}} e_{s(a)}\La,\ g)}&\Hom_\La(e_i\La,(1-e_i)\La)\ar@{>>}[d]^{(e_i\La,\ g)}\\
\Hom_\La({\displaystyle\bigoplus_{\begin{smallmatrix}a\in \overline{Q}_1, e(a)=i\end{smallmatrix}} e_{s(a)}\La},(1-e_i)E)\ar[r]^{\ \ \ \ \ \ \ \ \ (p_2,(1-e_i)E)}&\Hom_\La(e_i\La,(1-e_i)E).}
\]

It is clear that $\Hom_\La(\bigoplus_{\begin{smallmatrix}a\in \overline{Q}_1, e(a)=i\end{smallmatrix}} e_{s(a)}\La,g)$ and $\Hom_\La(e_i\La,g)$ are surjective since $e_{s(a)}\La$ and $e_i\La$ are projective. Moreover, by the sequence (\ref{simple sigma resolution}), $\Hom_\La(p_2,(1-e_i)\La)$ is also surjective. Consequently, $\Hom_\La(p_2,(1-e_i)E)$ is also surjective. 
Then the statement follows from the additivity of the functors. 
\end{proof}

Now we apply the above results to the following key proposition.

\begin{prop}\label{TI}
Let $T\in\langle I_1,\ldots,I_n\rangle$ and assume that $T$ is a basic support $\tau$-tilting ideal of $\La$. 
Then $I_iT$ is a basic support $\tau$-tilting $\Lambda$-module.
\end{prop}

\begin{proof}
There is nothing to show if $I_iT= T$. 
Assume that $I_iT\neq T$.
Then, by Lemma \ref{bongartz}, there exists a left mutation $\mu^-_{e_iT}(T)$.

Now let $e$ be an idempotent of $\La$ such that $T$ is a $\tau$-tilting $(\Lambda/\langle e\rangle)$-module. 
%We will apply Definition-Theorem \ref{air mutation} (iii) to the $\tau$-tilting $(\Lambda/\langle e\rangle)$-module $T$.
By applying the functor $-\otimes_\La T$ to (\ref{La I}), 
we get the following exact sequence of $(\Lambda/\langle e\rangle)$-module

\begin{eqnarray*}\label{mut sequence}
 e_i\La\otimes_\La T\overset{p_2\otimes_\La T}{\longrightarrow} \bigoplus_{\begin{smallmatrix}a\in \overline{Q}_1, e(a)=i\end{smallmatrix}} e_{s(a)}\La\otimes_\La T\overset{\pi\otimes_\La T}{\longrightarrow} e_iI_i\otimes_\La T\longrightarrow0.
\end{eqnarray*}

By applying Corollary \ref{tensor} to the support $\tau$-tilting left $\La$-module $T$, we have $e_iI_i\otimes_\La T\cong e_iI_iT$. Moreover, Lemma \ref{decompose} implies that $e_iI_iT$ is indecomposable. 
On the other hand, from Lemma \ref{approx}, $p_2\otimes_\La T$ is a left $(\add((1-e_i)T))$-approximation.

If $e_iI_iT=0$, then $p_2\otimes_\La T$ is clearly left minimal. 
Assume $e_iI_iT\neq0$. Since $\La$ is selfinjective, 
$e_jT$ has a simple socle for any $j\in Q_0$ and these simple modules are mutually non-isomorphic. Hence $e_{s(a)}\La\otimes_\La T$ and $e_iI_iT$ are not isomorphic for any $a\in \overline{Q}_1$ such that $e(a)=i$.
Thus, $\pi\otimes_\La T$ is a radical map and hence $p_2\otimes_\La T$ is left minimal. 
Then, applying Definition-Theorem \ref{air mutation} (iii), we conclude that $\mu^-_{e_iT}(T)\cong e_iI_i T\oplus(1-e_i)T=I_iT$ is a basic support $\tau$-tilting $(\Lambda/\langle e\rangle)$-module. 
\end{proof}

Now we are ready to prove Theorem \ref{main1}.

\begin{proof}[Proof of Theorem \ref{main1}]
We use induction on the number of products of the ideal $I_i$, $i\in Q_0$. 
%[Note that there are only finitely many pairwise ideals since $\La$ is finite dimensional.]
By Lemma \ref{I_i}, $I_i$ is a basic support $\tau$-tilting ideal of $\La$ for any $i\in Q_0$. 
Assume that $I_{i_1}\cdots I_{i_k}$ is a basic support $\tau$-tilting ideal for $i_1,\ldots,i_k\in Q_0$ and $k>1$. 
By Proposition \ref{TI}, $I_{i_{0}}(I_{i_1}\cdots I_{i_k})$ is a basic support $\tau$-tilting $\La$-module for any  vertex $i_{0}\in Q_0$. 
Similarly we can show that it is a basic support $\tau$-tilting left $\La$-module. 
\end{proof}

At the end of this subsection, we give the following lemma for later use.

\begin{lemm}\label{left mut}
Let $T\in\langle I_1,\ldots,I_n\rangle$. 
If $I_iT\neq T$, then there is a left mutation of $T:$ 
$$\mu^-_{e_iT}(T)\cong I_iT.$$
\end{lemm}
 
\begin{proof}
This follows from Lemma \ref{bongartz} and Proposition \ref{TI}. 
\end{proof}

Now we show examples of support $\tau$-tilting modules and their mutations.

\begin{exam}\label{exam1}
(a) Let $\La$ be the preprojective algebra of type $A_2$.  
In this case, $\HH(\sttilt\La)$ is given as follows.

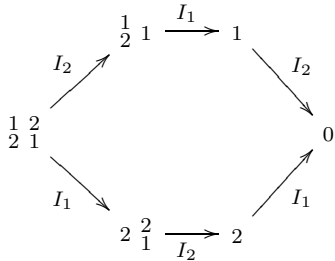
\[\xymatrix@C20pt@R20pt{&{\begin{smallmatrix}1\\ 2\end{smallmatrix}}{\begin{smallmatrix}\\ 1\end{smallmatrix}}\ar[r]^{I_1}&{\begin{smallmatrix}\\ 1\end{smallmatrix}}\ar[rd]^{I_2}&\\
{\begin{smallmatrix}1\\ 2\end{smallmatrix}}{\begin{smallmatrix}2\\1\end{smallmatrix}}\ar[ur]^{I_2}\ar[dr]_{I_1}&&&{\begin{smallmatrix}\\ 0\end{smallmatrix}}\\
&{\begin{smallmatrix}2\\ \end{smallmatrix}}{\begin{smallmatrix}2\\ 1\end{smallmatrix}}\ar[r]_{I_2}&{\begin{smallmatrix}\\ 2\end{smallmatrix}}\ar[ru]_{I_1}&}\]

Here we represent modules by their radical filtrations and we write a direct sum $X\oplus Y$ by $X\ Y$. 
Note that $I_i$ denotes a left multiplication.

(b) Let $\La$ be the preprojective algebra of type $A_3$. 
In this case, $\HH(\sttilt\La)$ is given as follows

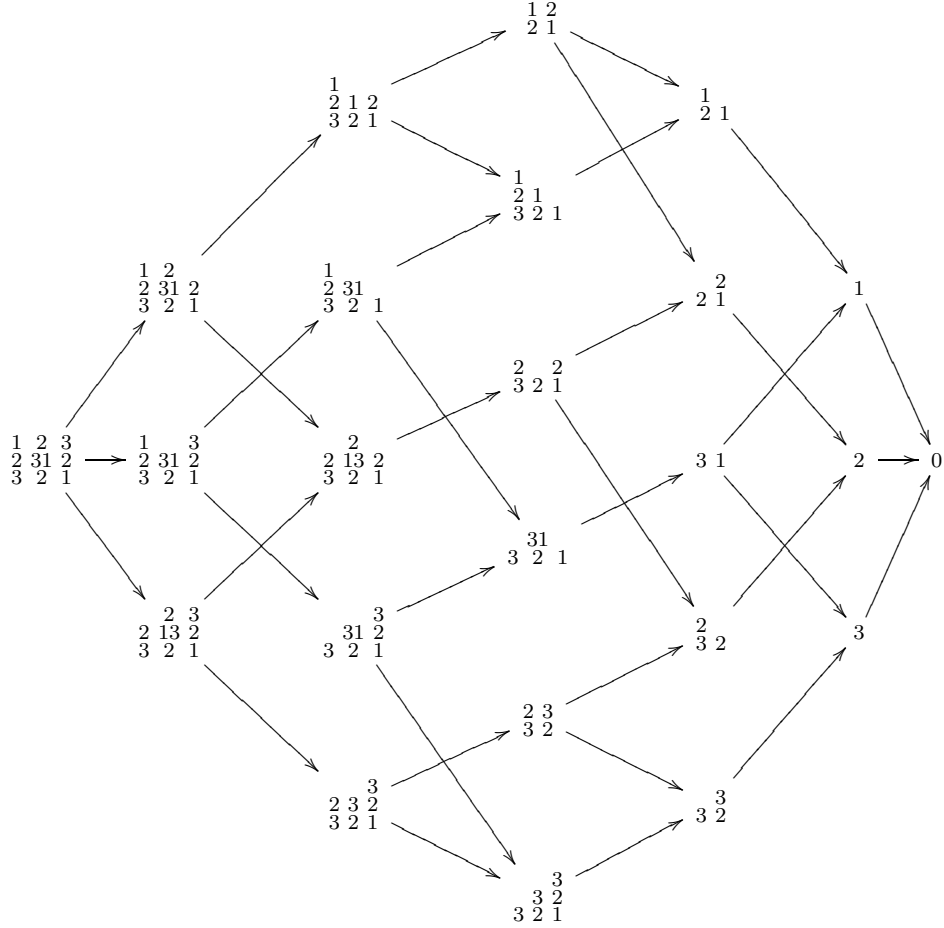
\[\xymatrix@C15pt@R10pt{
&   &   &   &   &    {\begin{smallmatrix} && \\ &1 &2 \\&2 &1  \end{smallmatrix}}  \ar[rrd]\ar[rrddd]   &    &   &   &   &     \\ 
&   &   & {\begin{smallmatrix}1 && \\2 &1 &2 \\3&2 &1  \end{smallmatrix}}\ar[rru]\ar[rrd]   &   &         &    & {\begin{smallmatrix} && \\ &1 & \\&2 &1  \end{smallmatrix}}\ar[rrdd]   &   &   &     \\ 
&   &   &   &   &    {\begin{smallmatrix} 1&& \\2 &1 & \\3&2 & 1 \end{smallmatrix}} \ar[rru]     &    &   &   &   &     \\ 
&  {\begin{smallmatrix}1 &2& \\ 2&31 &2 \\3&2 &1  \end{smallmatrix}} \ar[rruu]\ar[rrdd] &   & {\begin{smallmatrix} 1&& \\2 &31 & \\3&2 &1  \end{smallmatrix}}\ar[rru]\ar[rrddd]   &   &         &    & {\begin{smallmatrix} & \\  &2 \\2 &1  \end{smallmatrix}}\ar[rrdd]   &   & {\begin{smallmatrix} 1  \end{smallmatrix}}\ar[rdd]   &     \\ 
&   &   &   &   &   {\begin{smallmatrix} && \\2 & &2 \\3& 2&1  \end{smallmatrix}}\ar[rru]\ar[rrddd]       &    &   &   &   &     \\ 
{\begin{smallmatrix} 1&2&3 \\ 2&31 &2 \\3&2 &1  \end{smallmatrix}}\ar[uur]\ar[rdd]\ar[r] &  {\begin{smallmatrix} 1&&3 \\ 2&31 & 2\\3&2 &1  \end{smallmatrix}}\ar[rruu]\ar[rrdd]  &   &{\begin{smallmatrix} &2& \\ 2&13 &2 \\3&2 &1  \end{smallmatrix}}\ar[rru] &   &         &    & {\begin{smallmatrix} 3 &1  \end{smallmatrix}}\ar[rruu]\ar[rrdd]   &   & {\begin{smallmatrix}2  \end{smallmatrix}}\ar[r]   & {\begin{smallmatrix} 0\end{smallmatrix}}     \\ 
&   &   &   &   &    {\begin{smallmatrix} && \\ & 31& \\3&2 &1  \end{smallmatrix}}\ar[rru]      &    &   &   &   &     \\ 
&{\begin{smallmatrix} &2& 3\\ 2&13 &2 \\3&2 &1  \end{smallmatrix}}\ar[rruu]\ar[rrdd] &   & {\begin{smallmatrix} && 3\\ &31 &2 \\3&2 &1  \end{smallmatrix}}\ar[rru]\ar[rrddd]   &   &         &    & {\begin{smallmatrix} & \\ 2&  \\3&2  \end{smallmatrix}}\ar[rruu]   &   & {\begin{smallmatrix} 3 \end{smallmatrix}}\ar[ruu]   &     \\ 
&   &   &   &   &{\begin{smallmatrix} & \\ 2&3  \\3&2   \end{smallmatrix}} \ar[rru]\ar[rrd]     &    &   &   &   &     \\ 
&   &   & {\begin{smallmatrix} &&3 \\ 2& 3&2 \\3&2 &1  \end{smallmatrix}} \ar[rru]\ar[rrd] &   &         &    &{\begin{smallmatrix} & \\ & 3 \\3&2   \end{smallmatrix}}\ar[rruu]   &   &   &   \\
&   &   &   &   &   {\begin{smallmatrix} &&3 \\ &3 & 2\\3&2 &1  \end{smallmatrix}}\ar[rru]      &    &   &   &   &  
  }\] 
\end{exam}

In these two examples, $\HH(\sttilt\La)$ consists of a finite connected component. 
We will show that this is the case for preprojective algebras of Dynkin type in the sequel. Thus, all support $\tau$-tilting modules can be obtained by mutations from $\La$.

%%%%%%%%%%%%%%%%%%%%%%%%%%%%%%%%%%%%%%%%%%%%%%%%%%%%%%%%%%%%%%%%%%%%%%%%%%%%%%%%%%%%%%%%%%%%

\subsection{Bijection between support $\tau$-tilting modules and the Weyl group}\label{bijection with the Weyl group}
In the previous subsection, we have shown that elements of $\langle I_1,\ldots,I_n\rangle$ are basic support $\tau$-tilting ideals. In this subsection, we will show that any basic support $\tau$-tilting $\La$-module is given as an element of $\langle I_1,\ldots,I_n\rangle$. 
Moreover, we show that there exists a bijection between the Weyl group and basic support $\tau$-tilting $\La$-modules, which implies that there are only finitely many basic support $\tau$-tilting $\La$-modules. 

First we recall the following definition.

\begin{defi}
Let $Q$ be a finite connected acyclic quiver with vertices $Q_0=\{1,\ldots, n\}$. 
%For $i,j\in Q_0$, let $m_{ij}:=\sharp\{a\in Q_1|\ s(a)=i, e(a)=j\}+\sharp\{a\in Q_1|\ s(a)=j, e(a)=i\}.$ 
The \emph{Coxeter group} $W_Q$ associated to $Q$ is defined by the generators $s_1,\ldots , s_n$ and relations
\begin{itemize}
\item[$\bullet$] $s_i^2=1$,
\item[$\bullet$] $s_is_j=s_js_i$ if there is no arrow between $i$ and $j$ in $Q$,
\item[$\bullet$] $s_is_js_i=s_js_is_j$ if there is precisely one arrow between $i$ and $j$ in $Q$.
\end{itemize}

Each element $w\in W_Q$ can be written in the form $w=s_{i_1}\cdots s_{i_k}.$
If $k$ is minimal among all such expressions for $w$, then $k$ is called the \emph{length} of $w$ and we denote by $l(w)=k$. In this case, we call $s_{i_1}\cdots s_{i_k}$ a \emph{reduced expression} of $w$. 
\end{defi} 

In particular $W_Q$ is the \emph{Weyl group} if $Q$ is Dynkin. 
Then, we give the following important result. 

\begin{theorem}\label{main2}
Let $Q$ be a finite connected acyclic quiver with vertices $Q_0=\{1,\ldots, n\}$. 
There exists a bijection $W_Q\to\langle I_1,\ldots,I_n\rangle$. 
It is given by $w\mapsto I_w =I_{i_1}I_{i_2}\cdots I_{i_k}$ for any reduced 
expression $w=s_{i_1}\cdots s_{i_k}$.
\end{theorem}

The statement is given in \cite[III.1.9]{BIRS} for non-Dynkin quivers. The same result holds for 
Dynkin quivers. 
For the convenience of the reader, we give a complete proof of Theorem \ref{main2} for Dynkin quivers.

For a proof, we provide the following set-up. 

\begin{notation}\label{notation}
Let $\widetilde{Q}$ be an extended Dynkin quiver obtained from $Q$ by adding a vertex $0$ (i.e. $\widetilde{Q}_0=\{0\}\cup Q_0$) and the associated arrows. 
We denote by $\wLa$ the completion of the associated preprojective algebra and $\I_i:=\wLa(1-e_i)\wLa$ for $i\in\widetilde{Q}$. 
For each $w\in W_{\widetilde{Q}}$, let $\I_w:=\I_{i_1}\cdots \I_{i_k},$ where $w=s_{i_1}\cdots s_{i_k}$ is a reduced expression (this is well-defined \cite[Theorem III.1.9]{BIRS}). 
Note that we have 
$$\Lambda=\widetilde{\Lambda}/{ \langle e_0\rangle},\ \ I_i={\I_i}/{ \langle e_0\rangle}.$$
\end{notation}

\begin{proof}[Proof of Theorem \ref{main2}]
One can check that the map 
$$W_Q\to \langle I_1,\ldots,I_n\rangle,\ w=s_{i_1}\cdots s_{i_k}\mapsto I_w :=I_{i_1}I_{i_2}\cdots I_{i_k}$$ 
is well-defined and surjective by \cite[Theorem III.1.9]{BIRS}, where the assumption that $Q$ is non-Dynkin is not used.  
We only have to show the injectivity. 

Since ${i_1},\cdots,{i_k}\in Q_0$, we have $s_{i_j}\neq s_0$ for any $1\leq j\leq k$. 
Thus we have 
$$I_w=I_{i_1}\cdots I_{i_k}
=(\I_{i_1}/\langle e_0\rangle )\cdots (\I_{i_k}/\langle e_0\rangle) 
=(\I_{i_1}\cdots \I_{i_k})/\langle e_0\rangle=\I_w/\langle e_0\rangle.$$

Therefore, if $I_w=I_{w'}$ for $w,w'\in W_Q$, then we have $\I_w/\langle e_0\rangle=\I_{w'}/\langle e_0\rangle$. 
Since $s_0$ does not appear in reduced expressions of $w$ and $w'$, we have $e_0\in\I_w$ and $e_0\in\I_{w'}$. 
Hence, we obtain $\I_w=\I_{w'}$ and conclude $w=w'$ by \cite[Theorem III.1.9]{BIRS}.
\end{proof}

Thus, for any $w\in W_Q$, we have a support $\tau$-tilting module $I_w$. 
Now we give an explicit description of mutations using this correspondence. 

Let $(I_w,P_w)$ be a basic support $\tau$-tilting pair, where $P_w$ is a basic projective $\La$-module which is determined from $I_w$ (subsection \ref{Support tau tilting modules}). 
Because $e_i\La$ has simple socle $S_{\sigma(i)}$, where $\sigma:Q_0\to Q_0$ is a Nakayama permutation of $\La$,  
if $e_iI_w=0$, then $e_{\sigma(i)}\La\in\add P_w$ and hence $P_w$ is given by  
${\displaystyle\bigoplus_{\begin{smallmatrix}i\in Q_0,e_iI_w=0\end{smallmatrix}}}e_{\sigma(i)}\La$. 
For any $i\in Q_0$, we define a mutation 
$$\mu_i(I_w,P_w):=\left\{\begin{array}{cl} \mu_{(e_iI_w,0)}(I_w,P_w) & {\rm if}\ e_iI_w\neq 0\\ 
\mu_{(0,e_{\sigma(i)}\La)}(I_w,P_w)&{\rm if}\ e_iI_w= 0.\end{array}\right.
$$

Then we have the following result.

\begin{thm}\label{mutation}
For any $w\in W_Q$ and $i\in Q_0$, a mutation of basic support $\tau$-tilting $\La$-module $I_w$ is given as follows $:$
$$\mu_i(I_w,P_w)\cong (I_{s_iw},P_{s_iw}).$$
In particular, $W_Q$ acts transitively and freely on $\sttilt\La$ by 
$s_i(I_w,P_w):=\mu_i(I_w,P_w).$ 
\end{thm}

\begin{proof} 
By Theorem \ref{main1} and \ref{main2}, $I_{w}$ and $I_{s_iw}$ are basic support $\tau$-tilting modules, which are not isomorphic. 

Let $w=s_{i_1}\cdots s_{i_k}$ be a reduced expression.
If $l(s_iw)>l(w)$, then $s_iw=s_is_{i_1}\cdots s_{i_k}$ is a reduced
expression. 
Hence, %we have $I_iI_w=I_{s_iw}\neq I_{w}$ and 
we get $(1-e_i)I_{s_iw}=(1-e_i)I_{i}I_{w}=(1-e_i)I_{w}$. 
Therefore we have 
$$(I_{s_iw},P_{s_iw})\cong \left\{\begin{array}{cl} (e_iI_iI_w\oplus(1-e_i)I_w,P_w) & {\rm if}\ e_iI_iI_w\neq 0\\ 
((1-e_i)I_w,P_w\oplus e_{\sigma(i)}\La ) & {\rm if}\ e_iI_iI_w= 0.\end{array}\right.
$$ 
Thus $(I_{w},P_{w})$ and $(I_{s_iw},P_{s_iw})$ have a common almost complete support $\tau$-tilting pair $((1-e_i)I_w,P_w)$ as a direct summand. 

On the other hand, if $l(s_iw)<l(w)$, then $u:=s_iw$ satisfies $l(u)<l(s_iu)$. 
Then, similarly, we can show that $(I_{u},P_{u})=(I_{s_iw},P_{s_iw})$ and $(I_{s_iu},P_{s_iu})=(I_{w},P_{w})$ have a common almost complete support $\tau$-tilting pair as a direct summand. 

Therefore, by Definition-Theorem \ref{air mutation} (i), the first statement follows. 
The second statement given in this theorem follows from Theorem \ref{main2}.
\end{proof}

%We remark that $\mu_{(0,e_{\sigma(i)}\La)}(I_w,P_w)$ is always right mutation, whereas a right mutation and a left mutation happens for $\mu_{(e_iI_w,0)}(I_w,P_w)$. 
%We give a criterion when a left mutation (or right mutation) occurs in the next subsection.

Using the above result, we give the following crucial lemma.

\begin{lemm}\label{conn component}
The support $\tau$-tilting quiver $\HH(\sttilt\Lambda)$ has a finite connected component.
\end{lemm}

\begin{proof}
Let $C$ be the connected component of $\HH(\sttilt\Lambda)$ containing $\La$. 
By Theorem \ref{main1}, $I_w$ is a basic support $\tau$-tilting $\La$-module for any $w\in W_Q$.
We will show that $\{I_w\ |\ w\in W_Q\}$ gives the set of vertices of $C$.
%We proceed by induction on the length of $w\in W_Q$.  

First, we show that every $I_w$ belongs to $C$ by induction on the length of $w\in W_Q$. 
Assume that $\{I_w\ |\ w\in W_Q\ {\rm with}\ l(w)\leq k\}$ belong to $C$. 
Take $w\in W_Q$ with $l(w)= k+1$ and decompose $w=s_iw'$, where $l(w')= k$. 
Then, we have $I_w=I_iI_{w'}\neq I_{w'}$ by Theorem \ref{main2} and we get $\mu_{e_iI_{w'}}^-(I_{w'})\cong I_iI_{w'}$ by Lemma \ref{left mut}. Thus, $I_{w}$ belongs to $C$ and our claim follows inductively.

Next, we show that any vertex in $C$ has a form $I_w$ for some $w\in W_Q$. This follows the fact that the neighbours of $I_w$ are given as $I_{s_iw}$ for $i\in Q$ by Theorem \ref{mutation}. Therefore we obtain the assertion.  
\end{proof}

Now we recall the following useful result.

\begin{prop}\cite[Corollary 2.38]{AIR}\label{connect}
If $\HH(\sttilt\Lambda)$ has a finite connected component $C$, 
then $C=\HH(\sttilt\Lambda)$.
\end{prop}

As a conclusion, we can obtain the following statement.

\begin{thm}\label{spt contain}
Any basic support $\tau$-tilting $\La$-module is isomorphic to an element of $\langle I_1,\ldots,I_n\rangle$. In particular, it is isomorphic to a support $\tau$-tilting ideal of $\La$. 
\end{thm}

\begin{proof}
By Lemma \ref{conn component}, $\HH(\sttilt\Lambda)$ has a finite connected component, and Proposition \ref{connect} implies that it coincides with $\HH(\sttilt\Lambda)$. 
Thus, any support $\tau$-tilting $\La$-module is given by $I_w$ for some $w\in W_Q$. 
In particular, it is a support $\tau$-tilting ideal from Theorem \ref{main1}. 
\end{proof}

Furthermore, we give the following lemma. 

\begin{lemm}\label{right-two}
If right ideals $T$ and $U$ are isomorphic as $\La$-modules, then $T=U$.
\end{lemm}

\begin{proof}
Assume that $f:T\to U$ is an isomorphism of $\La$-modules. 
Since $\La$ is selfinjective, there exists a $\La$-module map $h:\La\to\La$ 
making the diagram 

\[\xymatrix{0\ar[r]&T\ar[r]^{\iota}\ar[d]^{f}&\La\ar@{.>}[d]^h\\
0\ar[r]&U\ar[r]^{\iota}&\La,}
\]
commutative, where $\iota$ is the canonical inclusion. 
Put $h(1)=\lambda\in\La$.   
By the commutative diagram, we have $\iota f(T)=h\iota(T)$ and hence we get 
$U=h(T)=\lambda T\subset T.$ 
By the similar argument for $f^{-1}$, we can show that $T\subset U.$ 
Thus we conclude that $T=U$.
\end{proof}

Now we are ready to prove the main result of this paper.

\begin{thm}\label{number}
There exist bijections between the isomorphism classes of basic support $\tau$-tilting $\La$-modules, basic support $\tau$-tilting $\La^{\op}$-modules and the elements of $W_Q$. 
\end{thm}

\begin{proof}By Theorem \ref{spt contain}, it is enough to show a bijection between $\sttilt\La$ and $W_Q$.

By Theorem \ref{main1} and \ref{main2}, we have a map 
$W_Q\ni w\mapsto I_w\in\sttilt\La$. This map is surjective since 
any support $\tau$-tilting $\La$-module is isomorphic to $I_w$ for some $w\in W_Q$ by 
Theorem \ref{spt contain}. 
Moreover, it is injective by Theorem \ref{main2} and Lemma \ref{right-two}. 
Thus we get the conclusion.
\end{proof}

\begin{remk}
It is known that the order of $W_Q$ is finite and the explicit number is determined by the type of the underlying graph of $Q$ as follows \cite[Section 2.11]{H}.

\[\begin{array}{|c|c|c|c|c|c|}
\hline                                                                       
\bar{\overline{Q}} & A_n &D_n & E_6 & E_7 & E_8 \\ \hline
  &(n+1)! &  2^{n-1}n!& 51840 & 2903040 & 696729600\\ \hline
\end{array}\] 
\end{remk}

At the end of this subsection, we give an application related to annihilators of $\sttilt\La$. Note that a tilting module is a \emph{faithful} $\tau$-tilting modules (i.e. its (right) annihilator vanishes) \cite[Proposition 2.2]{AIR}. 
We can completely determine annihilators of support $\tau$-tilting $\La$-modules. 
Here, we denote by $\ann X:=\{\lambda\in\La\ |\ X\lambda=0\}$ the annihilator of $X\in\mod\La$. 

\begin{cor}\label{annihi}
We have $\ann I_w=I_{w^{-1}w_0}$, where $w_0$ is the longest element in $W_Q$. 
Thus, we have a bijection 
$$\sttilt\La\to\sttilt\La,\ T\mapsto {\rm ann}T.$$ 
\end{cor}

For a proof, we recall the following useful result. 

\begin{prop}\cite[Proposition 6.4]{ORT}\label{ort dual}
Let $w_0$ be the longest element in $W_Q$. Then 
we have the following isomorphisms 
\begin{itemize}
\item[(a)]$DI_w\cong\La/I_{w^{-1}w_0}$ in $\mod\La^{\op}$.
\item[(b)]$DI_w\cong\La/I_{w_0w^{-1}}$ in $\mod\La$.
\end{itemize}
\end{prop}

Then we give a proof of Corollary \ref{annihi}.

\begin{proof}[Proof of Corollary \ref{annihi}]
It is clear that the right annihilator of $I_w$ coincides with the left annihilator of $DI_w$. 
On the other hand, we have isomorphism $DI_w\cong\La/I_{w^{-1}w_0}$ of $\La^{\op}$-modules by Proposition \ref{ort dual}. Thus, we get $\ann I_w=I_{w^{-1}w_0}$. The second statement given in 
this corollary follows from Theorem \ref{main2}.
\end{proof}

%%%%%%%%%%%%%%%%%%%%%%%%%%%%%%%%%%%%%%%%%%%%%%%%%%%%%%%%%%%%%%%%%%%%%%%%%%%%%%%%%%%%%%%%%%%%%%%%%%%%%%%%%%%%

\subsection{Partial orders of support $\tau$-tilting modules and the Weyl group}\label{Partial orders of}
In this subsection, we study a relationship between a partial order of support $\tau$-tilting modules and that of $W_Q$. In particular, we show that the bijection $W_Q\to\sttilt\La$ given in Theorem \ref{number} 
induces an (anti)isomorphism as partially ordered sets.

\begin{defi}\label{orders}
Let $u,w\in W_Q$. 
We write $u\leq_L w$ if there exist $s_{i_1},\ldots, s_{i_k}$ such that $w=s_{i_k}\ldots s_{i_1}u$ and $l(s_{i_j}\ldots s_{i_1}u)=l(u)+j$ for $0\leq j\leq k$. 
Clearly $\leq_L$ gives a partial order on $W_Q$, and we call $\leq_L$ the \emph{$($left$)$ weak order}. We denote by $\HH(W_Q,\leq_L)$ the Hasse quiver induced by  weak order on $W_Q$.
\end{defi}

The following assertion is immediate from the definition of the weak order.  

\begin{lemm}\label{right hasse}
An arrow in $\HH(W_Q,\leq_L)$ is given by 
\[\begin{array}{ll}
w\to s_iw&(\mbox{if $l(w)>l(s_iw)$}),\\
s_iw\to w&(\mbox{if $l(w)<l(s_iw)$})
\end{array}\]
for any $w\in W_Q$ and $i\in Q_0$. 
\end{lemm}

\begin{exam}
(a) Let $Q$ be a quiver of type $A_2$. 
Then $\HH(W_Q,\leq_L)$ is given as follows.

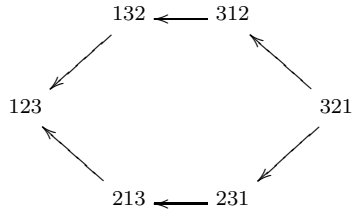
\[\xymatrix@C20pt@R20pt{&{\overset{132}{}}\ar[ld]&{\overset{312}{}}\ar[l]&\\
\overset{123}{}&&&{\overset{321}{}}\ar[ul]\ar[dl]\\
&{\overset{213}{}}\ar[lu]&{\overset{231}{}}\ar[l]&}\]

(b) Let $Q$ be a quiver of type $A_3$. 
Then $\HH(W_Q,\leq_L)$ is given as follows.

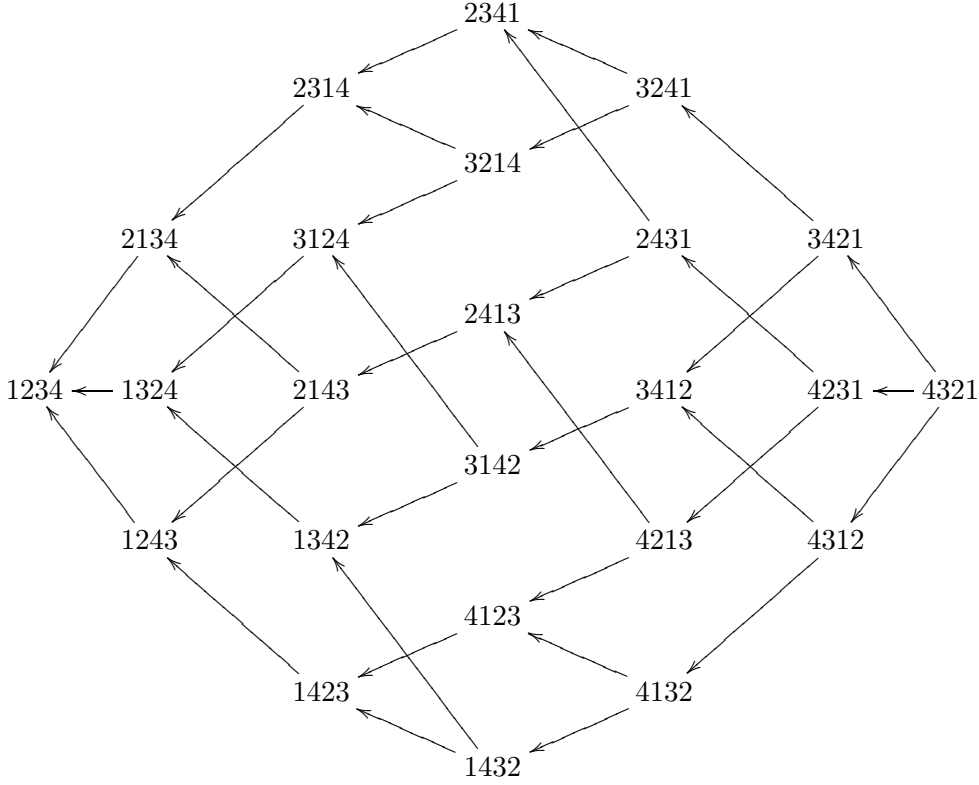
\[\xymatrix@C15pt@R15pt{
&   &   &   &   &    2341  \ar[lld]   &    &   &   &   &     \\ 
&   &   & 2314\ar[lldd]   &   &         &    & 3241\ar[lld]\ar[llu]   &   &   &     \\ 
&   &   &   &   &    3214 \ar[llu]\ar[lld]     &    &   &   &   &     \\ 
&  2134 \ar[ldd] &   & 3124\ar[lldd]   &   &         &    & 2431\ar[lld]\ar[lluuu]   &   & 3421\ar[lldd]\ar[lluu]   &     \\ 
&   &   &   &   &   2413\ar[lld]       &    &   &   &   &     \\ 
1234 &  1324\ar[l]  &   &2143\ar[lluu]\ar[lldd] &   &         &    & 3412\ar[lld]   &   & 4231\ar[lluu]\ar[lldd]   & 4321\ar[l]\ar[luu]\ar[ldd]     \\ 
&   &   &   &   &    3142\ar[lld]\ar[lluuu]      &    &   &   &   &     \\ 
&1243\ar[luu] &   & 1342\ar[lluu]   &   &         &    & 4213\ar[lld]\ar[lluuu]   &   & 4312\ar[lluu]\ar[lldd]   &     \\ 
&   &   &   &   &4123 \ar[lld]     &    &   &   &   &     \\ 
&   &   & 1423 \ar[lluu] &   &         &    &4132\ar[llu]\ar[lld]   &   &   &   \\
&   &   &   &   &   1432\ar[llu]\ar[lluuu]      &    &   &   &   &  }\]

\end{exam}

Next we give the following observation.

\begin{lemm}\label{length}
Let $w\in W_Q$ and $i\in Q_0$.
\begin{itemize}
\item[(i)]If $l(w)<l(s_iw)$, then $I_iI_w=I_{s_iw}\subsetneq I_{w}$ and we have a left mutation $\mu_i^-(I_w,P_w)$. 
\item[(ii)]If $l(w)>l(s_iw)$, then $I_{i}I_w=I_{w}\subsetneq I_{s_iw}$ and we have a right mutation $\mu_i^+(I_w,P_w)$. 
\end{itemize}
\end{lemm}

\begin{proof}
Let $w=s_{i_1}\cdots s_{i_k}$ be a reduced expression.
If $l(s_iw)>l(w)$, then $s_iw=s_is_{i_1}\cdots s_{i_k}$ is a reduced
expression. 
Thus, we have $I_iI_w=I_{s_iw}\subsetneq I_{w}$ by Theorem \ref{main2}, and we get a left mutation $\mu_i^-(I_w,P_w)$ by Lemma \ref{left mut}.

On the other hand, if $l(s_iw)<l(w)$, then $u:=s_iw$ satisfies $l(u)<l(s_iu)$. 
Then, similarly, we get a left mutation $\mu_i^-(I_u,P_u)=(I_{s_iu},P_{s_iu})=(I_w,P_w)$, 
or equivalently $\mu_i^+(I_w,P_w)=(I_{s_iw},P_{s_iw}).$
\end{proof}

%In Theorem \ref{mutation}, we give a description of mutations. 
Summarizing previous results,  we give the following equivalent conditions that a left (or right) mutation occurs. 

\begin{cor}\label{equivalent}
Let $w\in W_Q$ and $i\in Q_0$.
\begin{itemize}
\item[(a)] The following conditions are equivalent $:$
\begin{itemize}
\item[(i)] $l(w)<l(s_iw)$. 
\item[(ii)] $I_iI_w\neq I_{w}$.
\item[(iii)] We have a left mutation $\mu_i^-(I_w,P_w)$.
\end{itemize}
\item[(b)] The following conditions are equivalent $:$
\begin{itemize}
\item[(i)] $l(w)>l(s_iw)$. 
\item[(ii)] $I_iI_w= I_{w}$.
\item[(iii)] We have a right mutation $\mu_i^+(I_w,P_w)$.
\end{itemize}
\end{itemize}
\end{cor}

\begin{proof}
It is enough to show (a). 

By Lemma \ref{left mut} and \ref{length}, (i) implies (ii) and (ii) implies (iii). 
%Assume that $I_iI_w\neq I_{w}$. By Lemma \ref{left mut}, we get (iii). 
Assume that we have a left mutation $\mu_i^-(I_w,P_w)$. If $l(w)>l(s_iw)$, then we get a right mutation $\mu_i^+(I_w,P_w)$ by Lemma \ref{length}, which contradicts Definition-Theorem \ref{air mutation}. 
Thus (iii) implies (i).
\end{proof}

Finally we give the following results.

\begin{theorem}\label{main3}
The bijection in $W_Q\to\sttilt\La$ in Theorem \ref{number} gives an isomorphism of partially ordered sets 
$(W_Q,\leq_L)$ and $(\sttilt\La,\leq)^{\op}$.
\end{theorem}

\begin{proof}
It is enough to show that $\HH(W_Q,\leq_L)$ coincides with $\HH(\sttilt\La)^{\op}$ since the Hasse quivers determine the partial orders. 

By Theorem \ref{Hasse-mut} and Corollary \ref{equivalent}, arrows starting from $I_w$ are given by $I_w\to I_iI_w$ in $\HH(\sttilt\La)$, where $i\in Q_0$ satisfies $l(w)<l(s_iw)$.

On the other hand, Lemma \ref{right hasse} implies that arrows to $w$ are given by $s_iw\to w$ in $\HH(W_Q,\leq_L)$, where  $s_i$ satisfies $l(w)<l(s_iw)$. 
Therefore $\HH(\sttilt\La)^{\op}$ coincides with $\HH(W_Q,\leq_L)$ and the statement follows. 
\end{proof}

We remark that the \emph{Bruhat order} on $W_Q$ coincides with the reverse inclusion relation on $\langle I_1,\cdots I_n\rangle$ \cite[Lemma 6.5]{ORT}. 

%%%%%%%%%%%%%%%%%%%%%%%%%%%%%%%%%%%%%%%%%%%%%%%%%%%%%%%%%%%%%%%%%%%%%%%%%%%%%%%%%%%%%%%%%%%%%%%%%%%%%%%%%%%%

\section{$g$-matrices and cones}\label{Cone of $g$-vectors}
In this section, we study \emph{$g$-vectors} \cite{DK}, which is also called \emph{index} \cite{P} (see also \cite{AR1}), and \emph{$g$-matrices} of support $\tau$-tilting modules. 
We give a complete description of $g$-matrices in terms of reflection transformations. 
Moreover, we study cones defined by $g$-matrices, which provide a geometric realization of support $\tau$-tilting modules, and we show that they coincide with chambers of the associated root system. 

Throughout this section, unless otherwise specified, let $Q$ be a Dynkin quiver with vertices $Q_0=\{1,\ldots, n\}$, $\La$ the preprojective algebra of $Q$ and $I_i=\La(1-e_i)\La$ for $i\in Q_0$. 

First, we define $g$-vectors in our setting. We refer to \cite[section 5]{AIR} for a background of this notion.

\begin{defi}\label{def g-vector} 
By Krull-Schmidt theorem, we identify the set of isomorphism classes of projective $\La$-modules with submonoid $\mathbb{Z}^n_{\geq 0}$ of the free abelian group $\mathbb{Z}^n$. 

For a $\La$-module $X$, take a minimal projective presentation
\[\xymatrix{P_1(X)\ar[r]& P_0(X)\ar[r]&X \ar[r]&0}\]
and let 
$g(X)=(g_1(X),\cdots,g_n(X))^t:=P_0(X)-P_1(X)\in\mathbb{Z}^n$, where $t$ denotes the transpose matrix.
Then, for any $w\in W_Q$ and $i\in Q_0$, we define a $g$-\emph{vector} by  
$$\mathbb{Z}^n\ni g^i(w)=\left\{\begin{array}{cl} g(e_iI_w) & {\rm if}\ e_iI_w\neq 0\\ 
(0,\dots,0,\overset{\sigma(i)}{-1},0,\cdots,0)^t&{\rm if}\ e_iI_w= 0,\end{array}\right.
$$
where $\sigma:Q_0\to Q_0$ is a Nakayama permutation of $\La$ (i.e. $D(\La e_{\sigma(i)})\cong e_{i}\La$)). 

We define a \emph{$g$-matrix} of support $\tau$-tilting $\La$-module $I_w$ by 
$$g(w):=(g^1(w),\cdots,g^n(w))\in{GL}_n(\mathbb{Z})$$ 
and its \emph{cone} by 
$$C(w):=\{a_1g^1(w)+\cdots+a_ng^n(w)\ |\ a_i\in \mathbb{R}_{>0}\}.$$

Next, we define $g$-vectors for extended Dynkin cases. We follow Notation \ref{notation}. 
By Theorem \ref{main2}, there exists a bijection $$W_{\widetilde{Q}}\to\langle\I_0, \I_1,\ldots,\I_n\rangle,w\mapsto \I_w =\I_{i_1}\I_{i_2}\cdots \I_{i_k},$$ where $w=s_{i_1}\cdots s_{i_k}$ is a reduced 
expression of $w$. Note that $\I_w$ is a tilting $\wLa$-modules of projective dimension at most one \cite{IR} (hence $e_i\I_w\neq0$ for any $i\in \widetilde{Q}_0$). 
For any $w\in W_{\widetilde{Q}}$ and $i\in \widetilde{Q}_0$, 
take a minimal projective presentation of $\wLa$-module $e_i\I_w$ 
$$\xymatrix{P_1(e_i\I_w)\ar[r]& P_0(e_i\I_w)\ar[r]&e_i\I_w \ar[r]&0.}$$
Then we define a $g$-vector by 
$$\tg^i(w)=(g_0(e_i\I_w),\cdots,g_n(e_i\I_w))^t:=P_0(e_i\I_w)-P_1(e_i\I_w)\in\mathbb{Z}^{n+1},$$  
and we define a $g$-matrix by  
$$\tg(w):=(\tg^0(w),\cdots,\tg^n(w))\in{GL}_{n+1}(\mathbb{Z}).$$
\end{defi}

We use these notations in this section.  We also denote by $\overline{(-)}:=-\otimes_{\wLa}\La$ for simplicity. 

Then, we give the following results. 

\begin{prop}\label{minimal}
Let $w\in W_Q$ and $i\in Q_0$. %Assume that $e_iI_w\neq0$.   
Take a minimal projective presentation of $\wLa$-module  $e_i\I_w$ 
\begin{equation}\label{Iwseq1}
\xymatrix@C30pt@R30pt{0\ar[r]&P_1\ar[r]^{f_1}& P_0\ar[r]^{f_0}&e_i\I_w\ar[r]& 0.}
\end{equation}
Then, applying the functor $-\otimes_{\wLa}\La$ to $(\ref{Iwseq1})$, we have the following exact sequence 
%$\La$-module presentation of $e_iI_w$ 
\begin{equation*}\label{Iwseq4}
\xymatrix@C30pt@R30pt{0\ar[r]^{}&\ar[r]^{}\nu^{-1}(e_i\La/e_iI_w)&\overline{P_1}\ar[r]^{\overline{f_1}}& \overline{P_0}\ar[r]^{\overline{f_0}}&e_iI_w \ar[r]&0},
\end{equation*}
where $\nu^{-1}:=\Hom_\La(D\La,-).$ 
Moreover, if $e_iI_w\neq0$, then the sequence 
\begin{equation*}\label{Iwseq2}\xymatrix{\overline{P_1}\ar[r]^{\overline{f_1}}& \overline{P_0}\ar[r]^{\overline{f_0}}&e_iI_w \ar[r]&0}
\end{equation*}
is a minimal projective presentation of $\La$-module $e_iI_w$.
\end{prop}

\begin{proof}
Since $e_i\neq e_0$ and $\La=\widetilde{\Lambda}/{ \langle e_0\rangle}$, we have $e_i\I_w\otimes_{\wLa}\La\cong e_i\I_w/e_i\I_w\langle e_0 \rangle\cong e_i\I_w/e_i\langle e_0 \rangle\cong  e_i I_w$. 
Then, by applying the functor $-\otimes_{\wLa}\La$ to (\ref{Iwseq1}), we have the following exact sequence
$$\xymatrix{0\ar[r]^{}&\ar[r]^{\ \ \ \ \ \ \ g}\Tor_1^{\wLa}(e_i\I_w,\La)&\overline{P_1}\ar[r]^{\overline{f_1}}& \overline{P_0}\ar[r]^{\overline{f_0}}&e_iI_w \ar[r]&0}.$$

On the other hand, we have the following exact sequence
\begin{equation*}
\xymatrix@C30pt@R30pt{0\ar[r]&e_i\I_w\ar[r]^{}& e_i\wLa\ar[r]^{}&e_i\wLa/e_i\I_w\ar[r]& 0.}
\end{equation*}
%and $e_i\wLa/e_i\I_w\cong e_i\La/e_iI_w$. 

Then, we obtain $\La$-module isomorphisms 
\begin{eqnarray*}
\Tor_1^{\wLa}(e_i\I_w,\La)&\cong&D\Ext^1_{\wLa}(e_i\I_w,D\La)\ \ \ \ \ \ \ \   (\mbox{\cite{CE}})\\
&\cong& D\Ext^2_{\wLa}(e_i\wLa/e_i\I_w,D\La)\\
&\cong& \Hom_{\wLa}(D\La,e_i\wLa/e_i\I_w)\ \ \ \      (\mbox{2-CY\ property})\\
&\cong& \Hom_{\La}(D\La,e_i\La/e_iI_w)\ \ \ \      (e_i\wLa/e_i\I_w\cong e_i\La/e_iI_w)\\
&\cong& \nu^{-1}(e_i\La/e_iI_w).
\end{eqnarray*}

Now assume that $e_iI_w\neq0$. Then, clearly $\nu^{-1}(e_i\La/e_iI_w)$ is not projective and hence $g$ is a radical map. 
Moreover, it is clear that ${\overline{f_0}}$ is a projective cover by the assumption of $f_0$. Thus we obtain the conclusion. 
\end{proof}

\begin{remk}
It is known that, for every non-projective indecomposable module $X$ over a preprojective algebra of Dynkin type, its third syzygy $\Omega^3 X$ is isomorphic to $\nu^{-1}(X)$ \cite{AR2,BBK}.
This property appears in the above proposition.
\end{remk}

Using the above proposition, we give the following key result.

\begin{prop}\label{[Iw]=[Iw]}
For any $w\in W_Q$ and $i\in Q_0$, the $g$-vector 
$g^i(w)$ equals the image of $\tg^i(w)$ under the projection $\mathbb{Z}^{n+1}\to\mathbb{Z}^{n}$ which forgets the 0-th coordinate.
\end{prop}

\begin{proof}
Take a minimal projective presentation of $\wLa$-module $e_i\I_w$ 
\begin{equation}\label{sequence5}
\xymatrix@C30pt@R30pt{0\ar[r]&P_1\ar[r]^{f_1}& P_0\ar[r]^{f_0}&e_i\I_w\ar[r]& 0.}
\end{equation}

We decompose $P_0=P_0'\oplus (e_0\wLa)^\ell$ and $P_1=P_1'\oplus (e_0\wLa)^m$, where 
$(e_0\wLa)^\ell$ (respectively, $(e_0\wLa)^m$) is a maximal direct summand of $P_0$ (respectively, $P_1$) which belongs to $\add(e_0\wLa)$. 
For $\tg^i(w)=(g_0(e_i\I_w),g_1(e_i\I_w),\cdots,g_n(e_i\I_w))^t$, it is clear that 
$(g_1(e_i\I_w),\ldots,g_n(e_i\I_w))^t$ is determined by $P_0'$ and $P_1'$. 
We will show that it coincides with $g^i(w).$

By Proposition \ref{minimal}, applying $-\otimes_{\wLa}\La$ to (\ref{sequence5}), we have the following exact sequence 
\begin{equation}\label{Iwseq3}
\xymatrix@C30pt@R30pt{0\ar[r]&\nu^{-1}(e_i\La/e_iI_w)\ar[r]&\overline{P_1'}\ar[r]^{\overline{f_1}}& \overline{P_0'}\ar[r]^{\overline{f_0}}&e_iI_w\ar[r]& 0.}
\end{equation}
%where $\overline{P_0'}:=P_0'\otimes_{\wLa}\La$ and $\overline{P_1'}:=P_1'\otimes_{\wLa}\La$.

(i) Assume that $P_0'\neq0$. It implies that $\overline{P_0'}\neq0$ and hence $e_iI_w\neq0$. Therefore, by Proposition \ref{minimal}, the sequence (\ref{Iwseq3}) gives a minimal projective presentation of $e_iI_w$. Thus $g^i(w)$ coincides with $(g_1(e_i\I_w),\ldots,g_n(e_i\I_w))^t$. 

(ii) Assume that $P_0'=0$. 
It implies that $\overline{P_0'}=0$ and hence $e_iI_w=0$. Thus, $g^i(w)=(0,\dots,0,\overset{\sigma(i)}{-1},0,\cdots,0)^t$ by definition. 
On the other hand, by the sequence (\ref{Iwseq3}), we have 
$\overline{P_1'}\cong\nu^{-1}(e_i\La)\cong e_{\sigma(i)}\La$ and hence ${P_1'}\cong e_{\sigma(i)}\wLa$. Thus $\tg^i(w)=(g_0(e_i\I_w),0,\dots,0,\overset{\sigma(i)}{-1},0,\cdots,0)^t$ and we obtain the conclusion.
\end{proof}

Next we introduce the following notations \cite{Bou,BB}.  

\begin{defi}\label{geometric representation}
Let $Q$ be a finite connected acyclic quiver with vertices $Q_0=\{1,\ldots, n\}$. 
For $i,j\in Q_0$, let $m_{ij}:=\sharp\{a\in Q_1|\ s(a)=i, e(a)=j\}+\sharp\{a\in Q_1|\ s(a)=j, e(a)=i\}.$ 
Let ${\bf e}_1,\ldots,{\bf e}_n$ be the canonical basis of $\mathbb{Z}^n$.

The \emph{geometric representation} $\sigma: W_Q\rightarrow {GL}_n(\mathbb{Z})$ of $W_Q$ is defined by  
$$\sigma(s_i)({\bf e}_j)=\sigma_i({\bf e}_j):={\bf e}_j+(m_{ij}-2\delta_{ij}){\bf e}_i,$$
where $\delta_{ij}$ denotes the Kronecker delta.
On the other hand, the \emph{contragradient $\sigma^*: W_Q\rightarrow {GL}_n(\mathbb{Z})$ of the geometric representation}  is defined by 
$$\sigma^*(s_i)({\bf e}_j)=\sigma_i^*({\bf e}_j)=\left\{\begin{array}{cl} {\bf e}_j & i\neq j\\ -{\bf e}_j
+ {\displaystyle\sum_{\begin{smallmatrix}a\in \overline{Q}_1,
s(a)=i\end{smallmatrix}}} {\bf e}_{e(a)}& i=j.\end{array}\right.
$$

For an arbitrary expression $w=s_{i_1}\dots s_{i_k}$ of $w$, 
we define $\sigma(w):=\sigma_{i_k}\dots \sigma_{i_1}$ and $\sigma^*(w)=\sigma^*_{i_k}\dots \sigma^*_{i_1}$. 
\end{defi}

%We refer to \cite{Bou} for the following properties. 

%\begin{prop}\label{bb property}
%\begin{itemize}
%\item[(a)]The maps $\sigma:W_Q\rightarrow {GL}_n(\mathbb{Z})$ and $\sigma^*:W_Q\rightarrow {GL}_n(\mathbb{Z})$ are injective.
%\item[(b)]We have $(\sigma_i)^2=\id=(\sigma_i^*)^2$ for any $i\in Q_0$.
%\end{itemize}
%\end{prop}

Then, we give the following key result.% where we denote the set of $g$-matrices of support $\tau$-tilting $\La$-modules by $\gmat(\La)$.

\begin{prop}\label{act contra}
For any $w\in W_Q$, we have $g(w)=\sigma^*(w)$.
\end{prop}

\begin{proof}
By Proposition \ref{[Iw]=[Iw]}, the $g$-matrix $g(w)=(g^1(w),\cdots,g^n(w))$ is given by 1-st to $n$-th rows and columns of $\tg(w)=(\tg^0(w),\tg^1(w),\cdots,\tg^n(w))$.

On the other hand, by \cite[Theorem 6.6]{IR}, we have $\tg(w)=\tilde{\sigma}^*({w})$, 
where $\tilde{\sigma}^*$ is the contragradient $W_{\widetilde{Q}}\rightarrow {GL}_{n+1}(\mathbb{Z})$ of the geometric representation. 

For $w=s_{i_1}\dots s_{i_k}$, the 0-th column of $\tilde{\sigma}^*_{i_j}$ is $(\overset{}{1},\overset{}{0},\ldots,\overset{}{0})^t$ by the assumption of $s_{i_j}\neq s_0$ for any $1\leq j\leq k$.  
Because $1$-st to $n$-th rows and columns of $\tilde{\sigma}^*_{i_j}$ is equal to ${\sigma}^*_{i_j}$ by definition, we conclude that $g(w)=\sigma^*_{i_k}\dots \sigma^*_{i_1}=\sigma^*(w)$. 
%The second statement follows from Proposition \ref{bb property}.
\end{proof}

%\begin{remk}It is known that the $g$-vectors determine $\tau$-rigid pairs \cite{AIR}. In our case, this also follows from Proposition \ref{act contra} and \cite{BB}. \end{remk}

Finally we give the following settings. 

\begin{defi}
An element of $\Phi:=\{\sigma(w)({\bf e}_i)\}_{i\in Q_0,w\in W_Q}$ is called a \emph{root}. For $a,b\in \mathbb{R}^n$, we write a natural pairing $\langle a,b\rangle:=a\cdot b^t$.  
To each root $x\in\Phi$, we define a hyperplane 
$$H_x:=\{y\in \mathbb{R}^n\ |\ \langle y,x\rangle=0\}.$$
%and $H_x^+:=\{y\in \mathbb{R}^n\ |\ \langle y,x\rangle>0\}$. 
A \emph{chamber} of $\Phi$ is a connected component of $\mathbb{R}^n\setminus \bigcup_{x\in \Phi}H_x$. 
%\bigcap_{i\in Q_0} H^+_{{\bf e}_i}=
We denote by $C_0$ the chamber $\{a_1{\bf e}_1+\cdots+a_n{\bf e}_n \ |\ a_i\in \mathbb{R}_{>0}\}$. 
Note that chambers are given by 
$\coprod_{w\in W_Q} \sigma^*(w)(C_0)$ and there is a bijection between $W_Q$ and chambers via $w\mapsto \sigma^*(w)(C_0)$.  
\end{defi}

Then we obtain the following conclusion.

\begin{thm}\label{weyl group}
\begin{itemize}
\item[(a)] The set of $g$-matrices of support $\tau$-tilting $\La$-modules coincides with the subgroup $\langle \sigma^*_1,\ldots,\sigma^*_n\rangle$ of $\textrm{GL}(\mathbb{Z}^n)$ generated by $\sigma^*_i$ for all $i\in Q_0$. 
\item[(b)]For $w\in W_Q$, we have 
$$C(w)=\sigma^*(w)(C_0),$$
where $C(w)$ denotes the cone of $I_w$.
In particular, cones of basic support $\tau$-tilting $\La$-modules coincide with  chambers of the associated root system $\Phi$.
\end{itemize}
\end{thm}

\begin{proof}
From Proposition \ref{act contra}, (a) follows. Moreover, we have 

\begin{eqnarray*}
C(w)&=&\{a_1g^1(w)+\cdots+a_ng^n(w)\ |\ a_i\in \mathbb{R}_{>0}\}\\
&=& \{a_1\sigma^*(w)({\bf e}_1)+\cdots+a_n\sigma^*(w)({\bf e}_n) \ |\ a_i\in \mathbb{R}_{>0}\}\\
&=& \sigma^*(w)(C_0). 
\end{eqnarray*}
Thus, (b) follows.
\end{proof}

\begin{exam}
(a) Let $\La$ be the preprojective algebra of type $A_2$. 
In this case, the $g$-matrices of Example \ref{exam1} (a) are given as follows, where  
$\sigma_1^*=\left(\begin{smallmatrix}-1&0\\1 & 1 \end{smallmatrix}\right)$ and $\sigma_2^*=\left(\begin{smallmatrix}1&1\\0 & -1 \end{smallmatrix}\right)$.

\[\xymatrix@C20pt@R20pt{& {\begin{smallmatrix}1&1\\0 & -1 \end{smallmatrix}}\ar[r]^{\sigma_1^*}&
{\begin{smallmatrix}0&1\\-1 & -1 \end{smallmatrix}}\ar[rd]^{\sigma_2^*}&\\
{\begin{smallmatrix}1&0\\0 & 1 \end{smallmatrix}}
\ar[ur]^{\sigma_2^*}\ar[dr]_{\sigma_1^*}&&&{\begin{smallmatrix}\ 0&-1\\-1 &\ 0 \end{smallmatrix}}\\
&{\begin{smallmatrix}-1&0\\1 & 1 \end{smallmatrix}}\ar[r]_{\sigma_2^*}&
{\begin{smallmatrix}-1&-1\\1 & 0\end{smallmatrix}}\ar[ru]_{\sigma_1^*}&}\]

%For example, $g(s_1s_2)$ can be calculated by $\sigma^*(s_1s_2)=\sigma_2^*\sigma_1^*=\left(\begin{smallmatrix}1&1\\0 & -1 \end{smallmatrix}\right)\left(\begin{smallmatrix}-1&0\\1 & 1 \end{smallmatrix}\right)=\left(\begin{smallmatrix}0&1\\-1 & -1 \end{smallmatrix}\right)$.
Moreover, their chambers are given as follows.

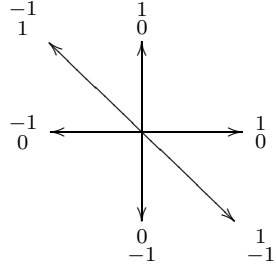
\[
\xymatrix{ {\begin{smallmatrix}-1\\1 \end{smallmatrix}}   & {\begin{smallmatrix}1\\0 \end{smallmatrix}} &\\
{\begin{smallmatrix}-1\\0 \end{smallmatrix}} \ar@{<->}[rr]^{}  &  &{\begin{smallmatrix}1\\0 \end{smallmatrix}}\\
 & {\begin{smallmatrix}0\\-1 \end{smallmatrix}} \ar@{<->}[uu] &\ar@{<->}[lluu]{\begin{smallmatrix}1\\-1 \end{smallmatrix}} }
\]

(b) Let $\La$ be the preprojective algebra of type $A_3$. 
In this case, the $g$-matrices of Example \ref{exam1} (b) are given as follows.  

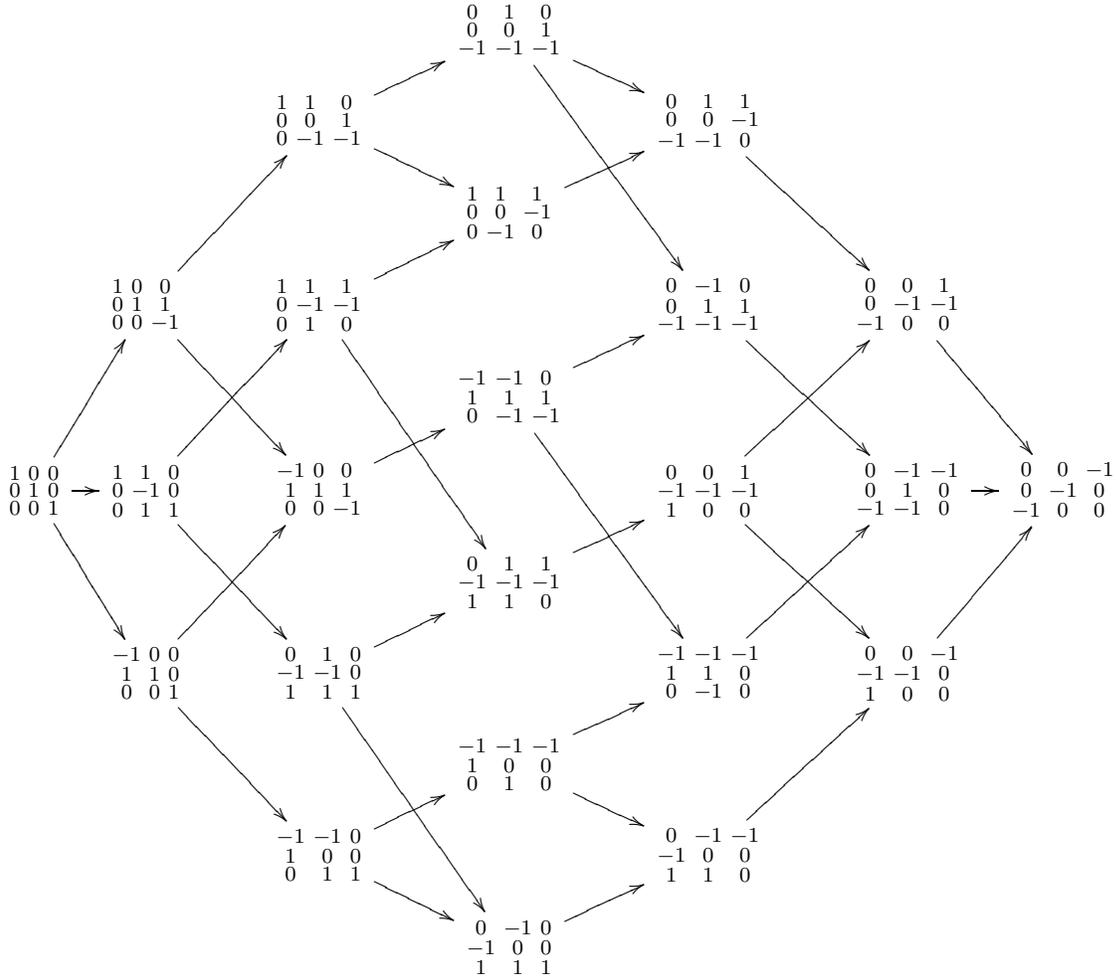
\[\xymatrix@C10pt@R8pt{
&   &   &   &   &    {\begin{smallmatrix}0 &1&0 \\ 0& 0&1 \\-1& -1&-1  \end{smallmatrix}}  \ar[rrd]\ar[rrddd]   &    &   &   &   &     \\ 
&   &   & {\begin{smallmatrix} 1&1&0 \\ 0&0 &1 \\0&-1 &-1  \end{smallmatrix}}\ar[rru]\ar[rrd]   &   &         &    & {\begin{smallmatrix} 0&1&1 \\ 0&0 &-1 \\-1&-1 &0  \end{smallmatrix}}\ar[rrdd]   &   &   &     \\ 
&   &   &   &   &    {\begin{smallmatrix} 1&1&1 \\0 & 0& -1\\0&-1 &0  \end{smallmatrix}} \ar[rru]     &    &   &   &   &     \\ 
&  {\begin{smallmatrix} 1&0&0 \\ 0&1 &1 \\0&0 &-1  \end{smallmatrix}} \ar[rruu]\ar[rrdd] &   & {\begin{smallmatrix}1 &1&1 \\0 &-1 &-1 \\0&1 &0  \end{smallmatrix}}\ar[rru]\ar[rrddd]   &   &         &    & {\begin{smallmatrix} 0&-1&0 \\0 &1 &1 \\-1&-1 &-1  \end{smallmatrix}}\ar[rrdd]   &   & {\begin{smallmatrix} 0&0&1 \\0 &-1 &-1 \\-1&0 &0  \end{smallmatrix}}\ar[rdd]   &     \\ 
&   &   &   &   &   {\begin{smallmatrix} -1&-1&0 \\1 &1 &1 \\0&-1 &-1  \end{smallmatrix}}\ar[rru]\ar[rrddd]       &    &   &   &   &     \\ 
{\begin{smallmatrix} 1&0&0 \\ 0&1 & 0\\0&0 &1  \end{smallmatrix}}\ar[uur]\ar[rdd]\ar[r] &  {\begin{smallmatrix} 1&1&0 \\ 0&-1 &0 \\0& 1&1  \end{smallmatrix}}\ar[rruu]\ar[rrdd]  &   &{\begin{smallmatrix}-1 &0&0 \\1 &1 &1 \\0&0 &-1  \end{smallmatrix}}\ar[rru] &   &         &    & {\begin{smallmatrix} 0&0&1 \\-1 &-1 &-1 \\1&0 &0  \end{smallmatrix}}\ar[rruu]\ar[rrdd]   &   & {\begin{smallmatrix}0 &-1&-1 \\0 &1 &0 \\-1&-1 &0  \end{smallmatrix}}\ar[r]   & {\begin{smallmatrix} 0&0&-1 \\0 &-1 &0 \\-1&0 &0  \end{smallmatrix}}     \\ 
&   &   &   &   &    {\begin{smallmatrix}0 &1&1 \\-1 &-1 &-1 \\1&1 &0  \end{smallmatrix}}\ar[rru]      &    &   &   &   &     \\ 
&{\begin{smallmatrix}-1 &0&0 \\1 &1 &0 \\0&0 &1  \end{smallmatrix}}\ar[rruu]\ar[rrdd] &   & {\begin{smallmatrix}0 &1&0 \\-1 &-1 &0 \\1&1 &1  \end{smallmatrix}}\ar[rru]\ar[rrddd]   &   &         &    & {\begin{smallmatrix}-1 &-1&-1 \\1 &1 &0 \\0&-1 &0  \end{smallmatrix}}\ar[rruu]   &   & {\begin{smallmatrix}0 &0&-1 \\-1 &-1 &0 \\1&0 &0  \end{smallmatrix}}\ar[ruu]   &     \\ 
&   &   &   &   &{\begin{smallmatrix} -1&-1&-1 \\1 &0 &0 \\0&1 &0  \end{smallmatrix}} \ar[rru]\ar[rrd]     &    &   &   &   &     \\ 
&   &   & {\begin{smallmatrix}-1 &-1&0 \\ 1& 0& 0\\0&1 &1  \end{smallmatrix}} \ar[rru]\ar[rrd] &   &         &    &{\begin{smallmatrix}0 &-1&-1 \\-1 &0 &0 \\1&1 &0  \end{smallmatrix}}\ar[rruu]   &   &   &   \\
&   &   &   &   &   {\begin{smallmatrix} 0&-1&0 \\ -1&0 &0 \\1&1 &1  \end{smallmatrix}}\ar[rru]      &    &   &   &   &  
  }\] 
\end{exam}

%%%%%%%%%%%%%%%%%%%%%%%%%%%%%%%%%%%%%%%%%%%%%%%%%%%%%%%%%%%%%%%%%%%%%%%%%%%%%%%%

\section{Further connections}\label{further}

In this section, we extend our bijections by combining with other works. We also explain some relationships with other results.  

We have the following bijections (We refer to \cite{BY} for the definitions (g), (h), (i) and (j)). 
\begin{thm}\label{further bijections}Let $Q$ be a Dynkin quiver with vertices $Q_0=\{1,\ldots, n\}$ and $\La$ the preprojective algebra of $Q$. 
There are bijections between the following objects.
\begin{itemize}
\item[(a)]The elements of the Weyl group $W_Q$.
\item[(b)]The set $\langle I_1,\ldots,I_n\rangle$.
\item[(c)]The set of isomorphism classes of basic support $\tau$-tilting $\La$-modules.
\item[(d)]The set of isomorphism classes of basic support $\tau$-tilting $\La^{\op}$-modules.
\item[(e)]The set of torsion classes in $\mod\La$.
\item[(f)]The set of torsion-free classes in $\mod\La$. 
\item[(g)]The set of isomorphism classes of basic two-term silting complexes in $\KKb(\proj\La)$. 
\item[(h)]The set of intermediate bounded co-$t$-structures in $\KKb(\proj\La)$ with respect to the standard co-$t$-structure.
\item[(i)]The set of intermediate bounded $t$-structures in $\DD^{\rm b}(\mod\La)$ with length heart with respect to the standard $t$-structure. 
\item[(j)]The set of isomorphism classes of two-term simple-minded collections in $\DD^{\rm b}(\mod\La)$.
\item[(k)]The set of quotient closed subcategories in $\mod KQ$.
\item[(l)]The set of subclosed subcategories in $\mod KQ$.
\end{itemize}
\end{thm}

We have given bijections between (a), (b), (c) and (d). 
Bijections between (g), (h), (i) and (j) are the restriction of \cite{KY} and it is given in \cite[Corollary 4.3]{BY} (it is stated for Jacobian algebras, but it holds for any finite dimensional algebra as they point out). 
Note that compatibilities of mutations and partial orders of these objects are discussed in \cite{KY}.

We will give bijections between (a), (e) and (f) by showing the following statement, which provides 
complete descriptions of torsion classes and torsion-free classes in $\mod\La$. 

\begin{prop}\label{bij torsion}
Any torsion class in $\mod\La$ is given as $\Fac I_w$ and any torsion-free class in $\mod\La$ is given as $\Sub \La/I_w$ for some $w\in W_Q$. 
Moreover, there exist bijections between the elements of $W_Q$, torsion classes and torsion-free classes. 
\end{prop}

\begin{proof}
Since there exists only finitely many support $\tau$-tilting modules by Theorem \ref{number}, 
Theorem \ref{iyama} implies that any torsion class (torsion-free class) is  functorially finite. Thus, Theorem \ref{fftor-spt} gives a bijection between basic support $\tau$-tilting $\La$-modules and torsion classes in $\mod\La$, which is  
given by the map $I_w\mapsto\Fac I_w$. 

Similarly, there exists a bijection between basic support $\tau$-tilting $\La^{\op}$-modules and torsion classes in $\mod\La^{\op}$. 
By the duality $D:\mod\La^{\op}\to\mod\La$, 
any torsion-free class of $\mod\La$ is given as $D(\Fac I_w)$, where $\Fac I_w$ is a torsion class in $\mod\La^{\op}$. 
Then, by Proposition \ref{ort dual}, 
we have $D(\Fac I_w)\cong\Sub (DI_w)\cong \Sub \La/I_{w_0w^{-1}},$ where $w_0$ is the longest element in $W_Q$. 
%From the above arguments, (iii) is clear.
\end{proof}

\begin{remk}%In \cite{BIRS}, several properties of \emph{2-CY} triangulated categories are  studied. From their results, 
%Since  are extension-closed functorially finite subcategories, 
It is shown that objects $\Fac I_w$ and $\Sub \La/I_w$ have several nice properties. 
For example, $\Fac I_w$ and $\Sub \La/I_w$ are \emph{Frobenius} and, moreover, \emph{stable 2-CY} categories which have cluster-tilting objects \cite{BIRS,GLS3}. 
\end{remk}

It is known that weak order of $W_Q$ is a \emph{lattice} (i.e. any two elements $x,y\in W_Q$ have a greatest lower bound, called \emph{meet} $x\bigwedge y$ and a least upper bound, called \emph{join} $x\bigvee y$). 
From the viewpoint of Theorem \ref{bij torsion}, 
it is natural to consider what is a join and meet in torsion classes. 
We give an answer to the question as follows. 

\begin{prop}\label{lattice}Let $\TT$ and $\TT'$ be torsion classes of $\mod\La$. 
Then we have $\TT\bigwedge\TT'=\TT\bigcap\TT'$ and $\TT\bigvee\TT'=\Tors\{\TT,\TT'\}$, where $\Tors\{\TT,\TT'\}$ is the smallest torsion class in $\mod\La$ containing $\TT$ and $\TT'$.
\end{prop}

\begin{proof}
It is clear that $\TT\bigcap\TT'$ and $\Tors\{\TT,\TT'\}$ are torsion classes. 
Since any torsion class in $\mod\La$ is functorially finite by Theorem \ref{iyama} and \ref{number}, the statement follows.
\end{proof}

Next we will see bijections between (a), (k) and (l), which are given by \cite{ORT}. We briefly recall their results and explain a relationship with our results. 
 
Let $X$ be a $\La$-module. We denote by $X_{KQ}$ the $KQ$-module by the restriction, that is, we forget the action of the arrows $a^*\in \overline{Q}$. 
Then we define $(-)_{KQ}:\mod\La\to\mod KQ, X\mapsto (X)_{KQ}:=\add X_{KQ}.$ 
%Define the subcategory $(X)_{KQ}$ of $\mod KQ$ whose indecomposable modules are those which appear as indecomposable summands of $X$ as a $KQ$-module, that is, $$(X)_{KQ}:=\add X_{KQ}.$$ 
Take $I_w$ for $w\in W_Q$. 
In \cite{ORT}, the authors show that $(I_w)_{KQ}$ (respectively, $(\La/I_w)_{KQ}$) is a quotient closed subcategory (respectively, subclosed subcategory) of $\mod KQ$, and any quotient closed subcategory (respectively, subclosed subcategory) is given in this form. 

Now, we can conclude that the functor $(-)_{KQ}$ gives a bijection between (e) and (k) (similarly (f) and (l)).   
Indeed, any torsion class is given by $\Fac I_w$ for some $w\in W_Q$ from Proposition \ref{bij torsion} and we have $(\Fac I_w)_{KQ}=(I_w)_{KQ}$ by a result of \cite{ORT}. 

Moreover, it is natural to consider when a quotient closed subcategory $(I_w)_{KQ}$ is a torsion class. 
The answer is given by \cite[Proposition 10.4]{ORT} and they show that elements of $W_Q$ satisfying a certain condition called being \emph{c-sortable} \cite{Re} are in bijection with torsion classes of $\mod KQ$ by this correspondence. 
We emphasize that there exist bijections between torsion classes of $\mod KQ$ and several important objects such as clusters in the cluster algebra given by $Q$ (we refer to \cite[Theorem 1.1]{IT} for more details). 
Thus, there exists an inclusion from the objects of \cite[Theorem 1.1]{IT} 
to the objects of Theorem \ref{further bijections}.

Finally, we provide a bijection between (c) and (g), which is given by \cite[Theorem 3.2]{AIR}, and explain a definition of $g$-vectors. 
Let $\KKb(\proj\La)$ be the homotopy category of bounded complexes of $\proj\La$. For a $\La$-module $X$, take a minimal projective presentation of $X$
\[\xymatrix{P_1(X)\ar[r]& P_0(X)\ar[r]&X \ar[r]&0.}\]
We denote by $P_X:=(\overset{-1}{P_1(X)}\to \overset{0}{P_0(X)})\in \KKb(\proj\La)$. 
Then, for a support $\tau$-tilting pair $(X,P)$ for $\La$, the map $(X,P)\mapsto P_X\oplus P[1]\in \KKb(\proj\La)$ gives a two-term silting complex in $\KKb(\proj\La)$ and it provides a bijection between (c) and (g). 
Note that the $g$-vector of $(X,P)$ is given as a class of the corresponding silting complex $P_X\oplus P[1]$ in the Grothendieck group $K_0(\KKb(\proj\La))$ \cite[section 5]{AIR}, and Definition \ref{def g-vector} is given by this correspondence.    
We also remark that, since $\La$ is selfinjective, it is shown that a $\nu$-\emph{stable} support $\tau$-tilting $\La$-module $X$ (i.e $X\cong\nu(X)$ for $\nu:=D\Hom_\La(-,\La)$) gives a two-term tilting complex by this correspondence \cite{M}. From Theorem \ref{number}, 
there exist only finitely many basic two-term tilting complexes in $\KKb(\proj\La)$.

%%%%%%%%%%%%%%%%%%%%%%%%%%%%%%%%%%%%%%%%%%%%%%%%%%%%%%%%%%%%%%%%%%%%%%%%%%%%%%%%%%%%%%%%%%%%

\end{document}